\documentclass[a4paper, 11pt,sumlimits,intlimits]{amsart}
\usepackage{amsmath,amsfonts,amsthm,amssymb,mathrsfs,latexsym,paralist, xypic, stmaryrd, soul}

\usepackage{tikz}
\usepackage{hyperref}

\usetikzlibrary{arrows,automata}

\tikzstyle{V}=[fill=black,circle,scale=0.4, outer sep = 4pt]

\allowdisplaybreaks

\newtheorem{thm}{Theorem}[section]
\newtheorem{prop}[thm]{Proposition}
\newtheorem{cor}[thm]{Corollary}
\newtheorem{lemma}[thm]{Lemma}

\theoremstyle{remark}
\newtheorem{rmk}[thm]{Remark}
\newtheorem{example}[thm]{Example}
\theoremstyle{definition}
\newtheorem{defn}[thm]{Definition}
\newtheorem{constr}[thm]{Construction}

\DeclareMathOperator{\Hom}{Hom}

\newcommand{\bi}{\begin{itemize}}
\newcommand{\ei}{\end{itemize}}
\newcommand{\be}{\begin{enumerate}}
\newcommand{\ee}{\end{enumerate}}

\newcommand{\K}{\mathcal{K}}
\newcommand{\M}{\mathcal{M}}

\newcommand{\R}{\mathbb{R}}
\newcommand{\N}{\mathbb{N}}
\newcommand{\Z}{\mathbb{Z}}

\newcommand{\chaincomplex}[1]{\left({#1}\right)}
\newcommand{\Presolution}{\mathcal{P}}

\begin{document}
\title[Isomorphism of the cubical and categorical cohomology]{Isomorphism of the cubical and categorical cohomology groups of a higher-rank graph}

\author{Elizabeth Gillaspy}\address{E.G.: Department of Mathematical Sciences,
	University of Montana,
	32 Campus Drive \#0864, 
	Missoula, MT 59812, USA
	}\email{elizabeth.gillaspy@mso.umt.edu}
\author{Jianchao Wu}\address{J.W.: Department of Mathematics, Penn State University, 109 McAllister Building, University Park, PA 16802, USA}\email{jianchao.wu@psu.edu}
\thanks{E.G. was partially supported by the Deutsches Forschungsgemeinschaft via the  SFB 878 ``Groups, Geometry, and Actions.'' J.W. was partially supported by NSF grant \#DMS--1564401.}

\date{\today}
\maketitle
\begin{abstract}
	We use category-theoretic techniques to provide two proofs showing that for a higher-rank graph $\Lambda$, its cubical \mbox{\mbox{(co-)}homology} and categorical \mbox{\mbox{(co-)}homology} groups are isomorphic in all degrees, thus answering a question of Kumjian, Pask and Sims in the positive. Our first proof uses the topological realization of a higher-rank graph, which was introduced by Kaliszewski, Kumjian, Quigg, and Sims. In our more combinatorial second proof, we construct, explicitly and in both directions, maps on the level of \mbox{\mbox{(co-)}chain} complexes that implement  said isomorphism. Along the way, we extend the definition of cubical \mbox{\mbox{(co-)}homology} to allow arbitrary coefficient modules. 
\end{abstract}

\tableofcontents

\section{Introduction}
The $C^*$-algebras associated to directed graphs \cite{CK, cuntz-CK, enomoto-watatani-1,kprr, kpr} have played an important role in $C^*$-algebra theory, largely because of the tight links between properties of the $C^*$-algebra, those of the underlying directed graph and those of the associated symbolic dynamical system.  For instance, the $K$-theory \cite{raeburn-szyman, drinen-tomforde} and the ideal structure \cite{bhrs,hong-syman} of graph $C^*$-algebras can be computed directly from the graph.  The close structural ties between directed graphs and their $C^*$-algebras can also be used to identify $C^*$-algebras which are not graph $C^*$-algebras: for example, any simple $C^*$-algebra which is neither AF nor purely infinite \textemdash\  such as the $C^*$-algebras of noncommutative tori \textemdash\  cannot be a graph $C^*$-algebra \cite{kpr}.  

Building on work of Robertson and Steger \cite{robertson-steger}, Kumjian and Pask introduced \emph{higher-rank graphs} in \cite{kp} to extend the successes of graph $C^*$-algebras to a broader class of $C^*$-algebras.   The higher-rank graphs of rank $k$ (also called $k$-graphs) can be viewed as a $k$-dimensional generalization of directed graphs (which correspond to the case $k=1$), although they are formally defined as a countable category equipped with a degree functor. The construction of $k$-graph $C^*$-algebras generalizes that of graph $C^*$-algebras. As in the case of directed graphs, many structural properties of $k$-graph $C^*$-algebras are evident from the underlying $k$-graphs, such as their simplicity and ideal structure \cite{rsy2, robertson-sims, skalski-zacharias-1, davidson-yang-periodicity, kang-pask}, quasidiagonality \cite{clark-huef-sims} and KMS states \cite{aHLRS, aHLRS1}.  Higher-rank graphs have also provided crucial examples \cite{ruiz-sims-sorensen, spielberg-kirchberg,bnrsw,bcfs,prrs}  for Elliott's  classification program for simple separable nuclear $C^*$-algebras.

Compared to the theory of graph $C^*$-algebras, a fascinating new feature of $k$-graphs with $k>1$ is the possibility to twist the construction of $C^*(\Lambda)$ with a 2-cocycle on $\Lambda$ \cite{kps-hlogy}, in a way that generalizes the construction of noncommutative tori. By expanding the class of higher-rank graph $C^*$-algebras to include twisted $k$-graph algebras, we vastly increase the class of $C^*$-algebras which we can analyze via the combinatorial perspective of higher-rank graphs.  For example, the irrational rotation algebras arise as twisted $k$-graph algebras \cite[Example 7.7]{kps-hlogy}, but not as (untwisted) graph or higher-rank graph algebras \cite[Corollary 5.7]{evans-sims}.\footnote{However,  \cite[Example 6.5]{prrs} shows that the irrational rotation algebras are Morita equivalent to certain 2-graph algebras.}

Along with this extra flexibility comes a series of questions. What type of 2-cocycles are allowed? When do 2-cocycles induce the same twisted $C^*$-algebra? How do we compute with them? What is the relation between this construction and that of twisted groupoid $C^*$-algebras? In order to answer these questions, a systematic study of the cohomology groups of higher-rank graphs is in order. 

Complicating the matter further is the fact that there is more than one construction of cohomology for a higher-rank graph. In view of the history of homological theories for topological spaces, categories, etc, this is not at all surprising \textemdash\ as the relation between singular and simplicial homology demonstrates, having multiple approaches can be a core strength of \mbox{(co-)}homology theories. For a higher-rank graph $\Lambda$ and a coefficient abelian group $M$, Kumjian, Pask, and Sims have defined both categorical cohomology groups $H^n_{} (\Lambda, M)$ \cite{kps-twisted} and cubical cohomology groups $H^n_{\operatorname{cub}} (\Lambda, M)$ \cite{kps-hlogy}. The latter can be viewed as the cohomology groups of the topological realization associated to a $k$-graph (\cite{kalisz-kumjian-quigg-sims, kps-hlogy}) and lead to cocycles which are often easy to compute explicitly. The former are computed from composable tuples by treating a higher-rank graph as a small category; they are more flexible and make some theoretical results easier to obtain. A natural question was raised: are the cubical and categorical cohomology groups isomorphic? 

In \cite{kps-twisted}, Kumjian, Pask, and Sims answered this question affirmatively in dimensions $n = 0$, $1$, and $2$. Furthermore, they provided explicit formulas for these isomorphisms on the cocycle level, making explicit computations possible. However, their proof methods were ad hoc and (in dimension 2) very technical. Although they conjectured that the cubical and categorical cohomology groups should agree in all dimensions, they suggested that a new approach would be needed.  

We remark that establishing isomorphism in all dimensions is desirable, even if one is only interested in 2-cocycles. For one thing, a proof that works in full generality will likely be more natural and give us a better understanding of these cohomology groups.  Perhaps more importantly, many crucial  techniques in homological algebra involve  long exact sequences \textemdash\ e.g., the long exact sequence associated to a change of coefficient groups \textemdash\ so an understanding of the entire collection of \mbox{(co-)}homology groups will be indispensable in applying these techniques to higher-rank graphs.  

In this paper, we provide two proofs that the cubical and categorical \mbox{(co-)}homology groups for a higher-rank graph $\Lambda$ do indeed agree in all dimensions. The first proof is conceptual but abstract, while the second one is computational and provides explicit chain maps. As predicted by Kumjian, Pask, and Sims, our proof methods rely on a new (or at least unusual) approach to $k$-graphs: we view $k$-graphs primarily as categories.  This contrasts with the usual combinatorial perspective on $k$-graphs, which views them as $k$-dimensional generalizations of directed graphs. We also make crucial use of the topological realizations of $k$-graphs, which were introduced in \cite{kalisz-kumjian-quigg-sims}.  
In addition to enhancing the \mbox{(co-)}homological tools available for the analysis of higher-rank graphs and their $C^*$-algebras, therefore, this paper demonstrates the utility of studying $k$-graphs from category-theoretic and topological perspectives.  Indeed, the authors believe that the insight offered by these perspectives should shed more light on various problems involving twisted higher-rank graph $C^*$-algebras, such as their simplicity  and $K$-theory; we plan to address these in future work.

This paper is organized as follows.  We begin by reviewing the basics of higher-rank graphs, as well as some concepts from homological algebra, among which is the somewhat less common notion of a \emph{free $\Lambda$-module over a base} (see Definition~\ref{def:free-module} or \cite{Lueck89}), which will be of great use to us. In Section~\ref{sec:cubical}, we review the construction of cubical \mbox{(co-)}homology and generalize it to allow arbitrary modules as coefficients. We work in this generality throughout the paper. 

Common to both of our proofs, in Section~\ref{sec:new-free-res}, we articulate our construction of a chain complex $\chaincomplex{\Z \widetilde{Q}_n(\widetilde{\Lambda})}_{n\in \N}$ of $\Lambda$-modules (see Construction~\ref{constr:cubical-resolution}), which we call the \emph{cubical free resolution} associated to a $k$-graph, though the fact that it is a resolution (i.e., it is exact) is only made clear later. This chain complex computes the cubical \mbox{(co-)}homology groups via standard constructions (see Proposition~\ref{prop:cubical-chain-iso}). Connecting it with the categorical \mbox{(co-)}homology is thus the central issue for the rest of the paper, for which our two proofs diverge. 

Our first proof of the isomorphism between the cubical and categorical \mbox{(co-)}homology groups of a $k$-graph $\Lambda$ builds on work of Kumjian and the first-named author \cite{gillaspy-kumjian}, which reinterprets the categorical  cohomology groups of \cite{kps-twisted} using the framework of modules over a small category.  Standard arguments (cf.~\cite[Corollary III.6.3]{maclane}) then imply that the categorical \mbox{(co-)}homology of $\Lambda$ can be computed by any free resolution of the trivial $\Lambda$-module $\Z^\Lambda$ \textemdash\ in particular, the cubical free resolution, provided that we can show it is indeed a resolution. This last point is the main goal of Section~\ref{sec:initial-vertices}, which offers a proof of the exactness of $\chaincomplex{\Z \widetilde{Q}_n(\widetilde{\Lambda})}_{n\in \N}$ by showing that if a $k$-graph contains an initial object in the category-theoretic sense, then its topological realization is contractible (see Proposition~\ref{prop:top-realiz-contr}). As a side comment, Remark~\ref{rmk:alg-of-Lambda-v} shows that whenever $\Lambda$ has an initial object, $C^*(\Lambda) $ is canonically isomorphic to the algebra of compact operators on the Hilbert space spanned by the vertices of $\Lambda$. 

In Section~\ref{sec:main}, we complete our first proof (see Theorem~\ref{thm:cubical-equals-categorical}) and discuss a few consequences. For example, our isomorphism implies that the categorical \mbox{(co-)}homology groups of a $k$-graph $\Lambda$ vanish in dimensions greater than $k$ and, at least when the coefficient is a constant module, only depend on the topological realization of $\Lambda$. These are not at all clear from the definition of categorical \mbox{(co-)}homology. 

The remaining Section~\ref{sec:chain-maps} details our more combinatorial second proof. In the same way as $\chaincomplex{\Z \widetilde{Q}_n(\widetilde{\Lambda})}_{n\in \N}$ computes the cubical \mbox{(co-)}homology, the categorical \mbox{(co-)}homology is defined in \cite{gillaspy-kumjian, kps-twisted} by a chain complex $\chaincomplex{\mathcal P_n}_{n\in \N}$ of $\Lambda$-modules, which may be called the \emph{simplicial free resolution}. Without using the knowledge that $\chaincomplex{\Z \widetilde{Q}_n(\widetilde{\Lambda})}_{n\in \N}$ is exact, we proceed by constructing explicit chain maps back and forth between $\chaincomplex{\Z \widetilde{Q}_n(\widetilde{\Lambda})}_{n\in \N}$ and $\chaincomplex{\mathcal P_n}_{n\in \N}$ that induce a chain homotopy equivalence between the two chain complexes (see Proposition~\ref{prop:homotopy-equivalence}); thus by standard homological algebra, these chain maps induce  isomorphisms between the two types of \mbox{(co-)}homology groups (see Theorem~\ref{thm:chain-maps}). Intuitively speaking, we construct a ``triangulation chain map'' $\triangledown_*$ which ``turns boxes into triangles'' (that is, converts cubical $n$-chains into categorical $n$-chains), as well as a ``cubulation chain map'' $\boxempty_*$ which ``turns triangles into boxes'' (that is, converts categorical $n$-chains into cubical $n$-chains). It takes some nontrivial computations to verify these indeed form chain maps, i.e., they intertwine the boundary maps in the chain complexes (see Proposition~\ref{prop:triangle-is-chain-map} and Theorem~\ref{thm:box-is-chain-map}). But once this is done, these chain maps can also be dualized to give cochain maps that induce isomorphisms between the cohomology groups, which enable us to convert cubical cocycles to categorical cocycles and vice versa. 

The final parts of Section \ref{sec:chain-maps} prove the naturality of these \mbox{(co-)}chain maps (see Proposition~\ref{prop:naturality}) and compare them with the explicit isomorphisms of \cite{kps-twisted} in degrees 0, 1, and 2. Up to a sign in degree 2, our isomorphisms agree with those of \cite{kps-twisted}.

\section{Preliminaries}

\subsection{Higher-rank graphs}\label{subsec:prelim-k-graphs}

We begin by fixing some notational conventions. The natural numbers $\N$ will always include 0; we write $e_i$ for the canonical $i$-th generator of $\N^k$. 
 If $n= (n_1, \ldots, n_k) \in \N^k$, we write $ |n| := \sum_{i=1}^k n_i.$

We often view $\N^k$ as a small category with one object, namely 0, and with composition of morphisms given by addition. Thus, the notation $n \in \N^k$ means that $n$ is a morphism in the category $\N^k$.  Inspired by this, we will follow the usual conventions for higher-rank graphs and use the arrows-only picture of category theory.  That is, we identify the objects of a small category $\Lambda$ with its identity morphisms, and 
\[ \lambda \in \Lambda \text{ means } \lambda \in \text{Mor}\, \Lambda.\]

Given $\lambda \in \Lambda$, we denote its source and range by $r(\lambda)$ and $s(\lambda)$   respectively.  For $r \geq 1$, the collection of composable $r$-tuples in $\Lambda$ is 
\[ \Lambda^{*r} = \{ (\lambda_1, \ldots, \lambda_r) \in \Lambda \times \cdots \times \Lambda \colon  s(\lambda_i) = r(\lambda_{i+1})\}.\]

\begin{defn}\cite[Definitions 1.1]{kp}
\label{defn:k-graph}
A \emph{higher-rank graph} of rank $k$, or a {\em $k$-graph}, is a countable category $\Lambda$ equipped with a \emph{degree functor} $d\colon  \Lambda \to \N^k$, satisfying the \emph{factorization property}: if $d(\lambda) = m +n$, there exist unique morphisms $\mu, \nu \in \Lambda$ with $d(\mu) =m,  d(\nu) = n$ and $\lambda = \mu \nu$.
We define $d_i(\lambda) \in \N $ such that $d(\lambda) = \sum_{i=1}^{k} d_i(\lambda) e_i \in \N^k $.  The set of all $k$-graphs constitutes the objects of the category $k \text{\textendash} \mathfrak{graph}$, whose morphisms are degree-preserving morphisms between $k$-graphs, called \emph{$k$-graph morphisms}.

We identify the objects of $\Lambda$ with $\Lambda^0 = \{\lambda \in \Lambda\colon  d(\lambda) = 0\},$ and often refer to them as the \emph{vertices} of $\Lambda$.  More generally, for $ n \in \N^k, \  \Lambda^n := d^{-1}(n)$, and for $v, w \in \Lambda^0$ we have 
\[ v\Lambda^n = \{ \lambda \in \Lambda^n\colon  r(\lambda) = v \} \quad \text{ and } \quad v\Lambda^n w = \{ \lambda \in v\Lambda^n\colon s(\lambda) = w\}.\]
The sets $v \Lambda, v \Lambda w, \Lambda^n v, \Lambda v$ are defined analogously.

An important example of a $k$-graph is $\Omega^k$, i.e., the poset category $(\mathbb{N}^k, \leq)$, with the degree map defined to be $d(m,n) = n-m \in \N^k$ for any $m,n \in \N^k$ with $m \leq n$. For any $m \in \N^k$, we define the $k$-graph $\Omega^k_{\leq m}$ to be the poset category $\left(\left\{ n \in \mathbb{N}^k \colon  n \leq m \text{ coordinatewise} \right\}, \leq \right)$, with the degree map defined in the same way. 
\end{defn}

\begin{rmk}\label{rmk:Lambda-m-n}
	Given any $k$-graph $\Lambda$ and a morphism $\lambda \in \Lambda$, the factorization property provides a canonical $k$-graph morphism $\Omega^k_{\leq d(\lambda)} \to \Lambda$ (cf.~\cite[Remarks 2.2]{kp}). We write $\lambda(n,m)$ for the image of $(n,m) \in \Omega^k_{\leq d(\lambda)}$ under this morphism. For example, we have $\lambda (0, 0) = r(\lambda)$, $\lambda (0, d(\lambda)) = \lambda$, and $\lambda (d(\lambda), d(\lambda)) = s(\lambda)$. 
\end{rmk}

\subsection{Homological algebra}

In this subsection, we recall some notions in homological algebra. We will generally follow the setting and terminologies of \cite[Section~9]{Lueck89} and \cite[Section~2]{Lueck02}, whereby the notion of a free module over a base (Definition~\ref{def:free-module}) is particularly useful for us. 

Throughout this subsection $\Lambda$ denotes a small category. To be consistent with the notations for higher-rank graphs, we may use $v \Lambda w$ to denote the set of morphisms from $w$ to $v$, where $v,w \in \operatorname{Obj}\Lambda$. 

\begin{defn}\label{def:lambda-module}
	A \emph{right (resp., left) $\Lambda$-module} is a contravariant (resp., covariant) functor from $\Lambda$ to the category $\mathfrak{AbGrp}$ of abelian groups. 
	For $\lambda \in \Lambda$, the effect of the group homomorphism $\mathcal{M}(\lambda)$ on an element $m$ in its domain is often denoted by $m \cdot \lambda$ for a right $\Lambda$-module and $\lambda \cdot m$ for a left $\Lambda$-module. 
	
	If $M$ is an abelian group, write $M^\Lambda$ for the left/right $\Lambda$-module with $M^\Lambda(v) = M$ for all $v \in \text{Obj}\, \Lambda$ and $M^\Lambda(\lambda) = \operatorname{id}_M$ for all $\lambda \in \text{Mor}\, \Lambda$.
	
	A \emph{morphism}  of (left/right) $\Lambda$-modules (also called a  \emph{$\Lambda$-module map}) $\eta\colon \mathcal{M} \to \mathcal{N}$ is thus a natural transformation $(\eta_v)_{v\in \operatorname{Obj}\Lambda}$. For the sake of brevity, we sometimes simply write $\eta$ for an individual homomorphism $\eta_v$. We say a $\Lambda$-module map $\eta$ is injective (respectively, surjective or bijective/isomorphic) if $\eta_v$ is so for every $v \in \operatorname{Obj}\Lambda$.

	The collection of all $\Lambda$-module maps $\eta\colon \mathcal{M} \to \mathcal{N}$ is denoted by $\operatorname{Hom}_\Lambda (\mathcal{M}, \mathcal{N})$. It forms an abelian group under pointwise addition. Notice that in the special case when $\Lambda$ is the category containing only one object and one morphism, then each $\Lambda$-module is just an abelian group and $\operatorname{Hom}_\Lambda (\mathcal{M}, \mathcal{N})$ recovers $\operatorname{Hom} (\mathcal{M}, \mathcal{N})$, the abelian group of all group homomorphisms from $\mathcal{M}$ to $\mathcal{N}$. 
	
	Given a right $\Lambda$-module $\mathcal{M}$ and a left $\Lambda$-module $\mathcal{N}$, the \emph{$\Lambda$-tensor product} $\mathcal{M} \otimes_\Lambda \mathcal{N}$ is the abelian group generated by the product sets $\mathcal{M}(v) \times \mathcal{N}(v)$ for $v \in \operatorname{Obj} \Lambda$ \textemdash~whose elements $(m , n)$ are called elementary tensors and denoted by $m \otimes n$ \textemdash~and subject to the relations 
	\begin{itemize}
		\item $m \otimes n + m' \otimes n = (m + m') \otimes n $ and $m \otimes n + m \otimes n' = m \otimes (n + n') $ for any $v \in \operatorname{Obj} \Lambda$, $m,m' \in \mathcal{M}(v)$ and $n, n' \in \mathcal{N}(v)$, and 
		\item $(m \cdot \lambda) \otimes n = m \otimes (\lambda \cdot n)$ for any $\lambda \in \Lambda$, $m \in \mathcal{M}(r(\lambda))$ and $n \in \mathcal{N}(s(\lambda))$.  
	\end{itemize}

	The $\operatorname{Hom}_\Lambda$ and $\otimes_\Lambda$ constructions are closely related. To see this, we notice that $\operatorname{Hom} (-, -)$ forms a functor from $\mathfrak{AbGrp}^{\operatorname{op}} \times \mathfrak{AbGrp} \to \mathfrak{AbGrp}$ (also called a \emph{bifunctor}; here ``op'' stands for the opposite category). Thus, if we fix a left $\Lambda$-module $\mathcal{N}$ and an abelian group $G$, we obtain a right $\Lambda$-module $\operatorname{Hom} (\mathcal{N}, G)_\Lambda$ via composition of functors.  To be precise, 
\[ \operatorname{Hom}(\mathcal{N}, G)_\Lambda(v) = \operatorname{Hom}(\mathcal N(v), G),\]
and the $\Lambda$-module structure is given by $ \varphi \cdot \lambda(n) := 	\varphi(\lambda \cdot n),$ where $\varphi \in \operatorname{Hom}(\mathcal N, G)_\Lambda(r(\lambda))$ and  $n \in \mathcal N(s(\lambda))$.
	 For any right $\Lambda$-module $\mathcal{M}$, there is then a natural isomorphism of abelian groups
	\begin{equation}\label{eq:tensor-hom}
		\operatorname{Hom}\left(\mathcal{M} \otimes_\Lambda \mathcal{N}, G \right) \cong \operatorname{Hom}_\Lambda\left( \mathcal{M}, \operatorname{Hom} (\mathcal{N}, G)_\Lambda \right) \; .
	\end{equation}
	We reiterate that the three Hom's have related but different meanings. 
\end{defn}

\begin{defn}\label{def:projective-module}
	We say that a sequence 
	\[ \mathcal{M} \to \mathcal{N} \to \mathcal{P}\]
	of $\Lambda$-modules is \emph{exact} if the induced sequence $ \mathcal{M}(v) \to \mathcal{N}(v) \to \mathcal{P}(v)$  of abelian groups is exact for all $v \in \text{Obj}\, \Lambda$.
	
	We say that a (left/right) $\Lambda$-module $\mathcal{P}$ is \emph{projective} if every surjective $\Lambda$-module map  $\mathcal{M} \to \mathcal{P}$ has a right inverse $\mathcal{P} \to \mathcal{M}$.
\end{defn}

\begin{rmk}\label{rmk:module-over-ring}
	We point out that $\Lambda$-modules may also be viewed as modules over the ring $\Z \Lambda$. More precisely, the ring $\Z \Lambda$ is defined as the free abelian group generated by the morphisms in $\Lambda$ and equipped with the multiplication 
	\[
	(n \lambda) \cdot (m \mu) = 
	\begin{cases}
	(n m) (\lambda \mu) \, , & s(\lambda) = r(\mu) \\
	0 \, , & s(\lambda) \not= r(\mu)
	\end{cases}
	\]
	for any $\lambda, \mu \in \Lambda$ and $m,n \in \mathbb{Z}$. Any right $\Lambda$-module $\mathcal{M}$ gives rise to the right $\Z \Lambda$-module $|\mathcal{M}| := \bigoplus_{v \in \operatorname{Obj}\Lambda} \mathcal{M}(v)$, where the module structure is defined so that for any $v \in \operatorname{Obj}\Lambda$, $a \in \mathcal{M}(v)$, $n \in \mathbb{Z}$ and $\lambda \in \Lambda$, we have
	\[
	a \cdot (n \lambda) = 
	\begin{cases}
	n (a \cdot \lambda) \in \mathcal{M}(s (\lambda)) \, , & v = r(\lambda) \\
	0  \, , & v \not= r(\lambda)
	\end{cases}
	\; .
	\] 
	This assignment is unique since we may recover $\mathcal{M}$ from $|\mathcal{M}|$ by setting $\mathcal{M}(v) := |\mathcal{M}| \cdot (1v)$ for  $v \in \operatorname{Obj}\Lambda$. 
	A similar identification works for left $\Lambda$-modules. 
\end{rmk}

Thus we may think of a $\Lambda$-module $\mathcal M$ as the single abelian group $|\mathcal M|$ instead of a family of abelian groups, which has the advantage of simplifying some notations. 

The free $\Lambda$-modules which we now describe will play a central role in our arguments in this paper. 

\begin{defn}\label{def:free-module}	
	Let $\mathcal{P}$ be a right (resp., left) $\Lambda$-module and $(B_v)_{v \in \operatorname{Obj}}$ be a collection of sets where each $B_v$ is a set of elements in the abelian group $\mathcal{P}(v)$. 
	Then $\mathcal{P}$ is said to be \emph{free over $(B_v)_{v \in \operatorname{Obj}}$} and $(B_v)_{v \in \operatorname{Obj}}$ is called a \emph{base} for $\mathcal P$ 
	if for any right (resp., left) $\Lambda$-module $\mathcal{Q}$ and any tuple of maps $\left(\varphi_v\colon B_v \to \mathcal{Q}(v)\right)_{v\in \operatorname{Obj}\Lambda}$, there is exactly one $\Lambda$-module map $\Phi \colon \mathcal{P} \to \mathcal{Q}$ extending $(\varphi_v)_{v \in \operatorname{Obj}\Lambda}$, that is, for any $v \in \operatorname{Obj}\Lambda$ and $b \in B_v$, we have $\Phi(v)(b) = \varphi_v (b)$. 
	
	To simplify notation, we write $B = \bigsqcup_{v \in \operatorname{Obj}\Lambda}  B_v$ and often say $\mathcal{P}$ is \emph{free over $B$} or $B$ is a base for $\mathcal P$, when there is no danger of confusion. 
\end{defn}

We now describe a few alternate characterizations of free $\Lambda$-modules. For the sake of simplicity, we restrict ourselves to right $\Lambda$-modules, though analogous statements can be made for left $\Lambda$-modules. 

\begin{rmk}\label{rmk:free-module-isomorphisms}
	It is easy to see that in fact, a right $\Lambda$-module $\mathcal{P}$ is free over a base $B$ if and only if there is a \emph{natural} bijection
	\begin{equation}\label{eq:free-module-hom}
		\operatorname{Hom}_\Lambda \left(\mathcal{P}, \mathcal{M} \right) \overset{\cong}{\to} \prod_{v \in \operatorname{Obj}\Lambda} \prod_{b \in B_v} \mathcal{M}(v) \, , \quad \eta \mapsto \left( \eta_v (b)  \right)_{v \in \operatorname{Obj}\Lambda, ~ b \in B_v}  
	\end{equation}
	for all right $\Lambda$-modules $\mathcal{M}$. Note that this map is always a group homomorphism. 
	
	It thus follows from Equation~\eqref{eq:tensor-hom} that for a right $\Lambda$-module $\mathcal{P}$ which is free over a base $B$ and an arbitrary left $\Lambda$-module $\mathcal{N}$, there is a natural group isomorphism
	\begin{equation}\label{eq:free-module-tensor}
		\bigoplus_{v \in \operatorname{Obj}\Lambda} \bigoplus_{b \in B_v} \mathcal{N}(v) \overset{\cong}{\to}  \mathcal{P} \otimes_\Lambda \mathcal{N} \, , \quad  \left( a_{v,b} \right)_{v \in \operatorname{Obj}\Lambda, ~ b \in B_v} \mapsto \sum_{v \in \operatorname{Obj}\Lambda} \sum_{b \in B_v} b \otimes a_{v,b} \; .
	\end{equation}
\end{rmk}

Next we give an intrinsic description of free modules, using the following ``hom module'' as a building block.

\begin{defn}\label{def:Z-lambda-star}
Given $v \in \operatorname{Obj} \Lambda$, we denote by $\Z v\Lambda \ast$ the right $\Lambda$-module assigning,  
\begin{itemize}
	\item to each $w \in \operatorname{Obj} \Lambda$, the free abelian group $\Z v\Lambda w$ generated by the set $v\Lambda w$ of morphisms in $\Lambda$ with range $v$ and source $w$, and, 
	\item to each morphism $\lambda \in \Lambda$, the homomorphism $\Z v\Lambda r(\lambda) \to \Z v\Lambda s(\lambda)$ induced by right multiplication by $\lambda$. 
\end{itemize}
\end{defn}

\begin{rmk}\label{rmk:free-module-Yoneda}
	Given any right $\Lambda$-module $\mathcal{M}$ and subsets $B_v \subset \mathcal{M}(v)$ for $v \in \operatorname{Obj} \mathcal M$, define a right $\Lambda$-module 
	\begin{equation}\label{eq:canonical-free-module}
		\mathcal F := \bigoplus_{v\in \operatorname{Obj}\Lambda} \bigoplus_{b \in B_v} \Z v\Lambda\ast.
	\end{equation}
	We view each $B_v$ as a subset of $\mathcal F(v) = \bigoplus_{b \in B_v} \Z v\Lambda v$ by identifying each $b \in B_v$ with $1v$ in the corresponding summand. With this convention, one immediately sees that $\mathcal F$ is free over the base $(B_v)_{v \in \operatorname{Obj}}$ and thus we have a canonical $\Lambda$-module map from $\mathcal F$ to $\mathcal{M}$ that fixes $(B_v)_{v \in \operatorname{Obj}}$. More explicitly, this $\Lambda$-module map is given by
	\begin{align*}
		\mathcal F(w) = \bigoplus_{v\in \operatorname{Obj}\Lambda} \bigoplus_{b \in B_v} \Z v\Lambda w  &\to \mathcal{M} (w) \\
		(v, b, \lambda) &\mapsto b \cdot \lambda \; ,
	\end{align*}
	where the tuple $(v, b, \lambda)$, with $\lambda \in v \Lambda w$, denotes a typical generator in the summand indexed by $v \in \operatorname{Obj}\Lambda$ and $b \in B_v$. It follows from the Yoneda lemma that $\mathcal{M}$ is free over $(B_v)_{v \in \operatorname{Obj}}$ if and only if $\mathcal{M}$ is isomorphic to $\mathcal F$ via this $\Lambda$-module map. 
\end{rmk}

\begin{rmk}\label{rmk:free-module-projective}
	It is a direct consequence of Definition~\ref{def:free-module} that if a right $\Lambda$-module $\mathcal{P}$ is free over a base $B$, then it is projective. Indeed, for any surjective $\Lambda$-module map $\eta \colon \mathcal{M} \to \mathcal{P}$, any partial lift $B_v \to \mathcal{M}(v)$ of $\eta_v$ gives rise to a right inverse $\Lambda$-module map $\mathcal{P} \to \mathcal{M}$.
\end{rmk}

\begin{rmk}
	The reader may be familiar with the notion of freeness for modules over a ring $R$, i.e., being isomorphic to the direct-sum module $\bigoplus_{s \in S} R$ over a generating set $S$. This is less general than the above notion of free $\Lambda$-modules, in the sense that when we consider a free $\Lambda$-module as a module over the ring $\Z \Lambda$ as in Remark~\ref{rmk:module-over-ring}, it may not be free. For example, for any $v \in \operatorname{Obj}\Lambda$, the free right $\Lambda$-module $\Z v\Lambda \ast$ in Definition~\ref{def:Z-lambda-star} corresponds to the right $\Z \Lambda$-module $v (\Z \Lambda)$, that is, the right ideal generated by the idempotent $v \in \Z \Lambda$. This $\Z \Lambda$-module is projective but not free when $\Lambda$ has more than one object. 
\end{rmk}

In order to work with \mbox{(co-)}homology, we also need the basics of \mbox{(co-)}chain complexes of $\Lambda$-modules. 

\begin{defn}
	A \emph{chain complex} $\mathcal{C}_*$ of (left/right) $\Lambda$-modules is a bi-infinite sequence
	\[
		\ldots \overset{\partial_{n+1}}{\to} \mathcal{C}_{n} \overset{\partial_{n}}{\to} \mathcal{C}_{n-1} \overset{\partial_{n-1}}{\to} \ldots \overset{\partial_{1}}{\to} \mathcal{C}_{0} \overset{\partial_{0}}{\to} \mathcal{C}_{-1} \overset{\partial_{-1}}{\to} \ldots 
	\]
	of $\Lambda$-modules and $\Lambda$-module maps such that $\partial_{n} \circ \partial_{n+1} = 0$ for all $n$. Dually, a \emph{cochain complex} $\mathcal{C}^*$ of (left/right) $\Lambda$-modules is a bi-infinite sequence
	\[
	\ldots \overset{\delta_{n+1}}{\leftarrow} \mathcal{C}^{n+1} \overset{\delta_{n}}{\leftarrow} \mathcal{C}^{n} \overset{\delta_{n-1}}{\leftarrow} \ldots \overset{\delta_{0}}{\leftarrow} \mathcal{C}^{0} \overset{\delta_{-1}}{\leftarrow} \mathcal{C}^{-1} \overset{\delta_{-2}}{\leftarrow} \ldots 
	\]
	of $\Lambda$-modules and $\Lambda$-module maps such that $\delta_{n+1} \circ \delta_{n} = 0$ for all $n$.  
	We also allow the index $n$ to range over only a subset of $\mathbb{Z}$ (e.g., typically, $\mathbb{N}$), in which case, it is understood that $\mathcal{C}_n = 0$ or $\mathcal{C}^n = 0$ wherever it is not specified. A \mbox{(co-)}chain complex is \emph{acyclic} if the sequence is exact at every place. 
	
	Given two chain complexes $\mathcal{C}_*$ and $\mathcal{D}_*$ of $\Lambda$-modules, a \emph{chain map} $f_*$ from $\mathcal{C}_*$ to $\mathcal{D}_*$ consists of a sequence of $\Lambda$-module maps $f_n \colon  \mathcal{C}_n \to \mathcal{D}_n$ that makes the following diagram commutative
	\[
		\xymatrix{
				\ldots \ar[r]^{\partial_{n+2}} & \mathcal{C}_{n+1} \ar[d]^{f_{n+1}} \ar[r]^{\partial_{n+1}} & \mathcal{C}_{n} \ar[d]^{f_{n}} \ar[r]^{\partial_{n}} & \mathcal{C}_{n-1} \ar[d]^{f_{n-1}} \ar[r]^{\partial_{n-1}} & \ldots \\
				\ldots \ar[r]^{\partial_{n+2}} & \mathcal{D}_{n+1} \ar[r]^{\partial_{n+1}} & \mathcal{D}_{n} \ar[r]^{\partial_{n}} & \mathcal{D}_{n-1} \ar[r]^{\partial_{n-1}} & \ldots \\
			}
	\]
	We say two chain maps $f_*, g_* \colon \mathcal{C}_* \to \mathcal{D}_*$ are \emph{homotopic} and write $f_* \sim_{\operatorname{h}} g_*$ if there is a sequence of $\Lambda$-module maps $h_n \colon  \mathcal{C}_n \to \mathcal{D}_{n+1}$ such that $f_n - g_n = \partial_{n+1} h_n + h_{n-1} \partial_n$. Two chain complexes $\mathcal{C}_*$ and $\mathcal{D}_*$ are \emph{homotopy equivalent} if there are chain maps $f_* \colon \mathcal{C}_* \to \mathcal{D}_*$ and $g_* \colon \mathcal{D}_* \to \mathcal{C}_*$ such that $f_* \circ g_* \sim_{\operatorname{h}} \operatorname{id}_{\mathcal{D}_*}$ and $g_* \circ f_* \sim_{\operatorname{h}} \operatorname{id}_{\mathcal{C}_*}$. In particular, a chain complex $\mathcal{C}_*$ is \emph{contractible} if it is homotopy equivalent to the zero chain complex, or equivalently, if there are $\Lambda$-module maps $h_n \colon  \mathcal{C}_n \to \mathcal{C}_{n+1}$ such that $\operatorname{id}_{\mathcal{C}_*} = \partial_{n+1} h_n + h_{n-1} \partial_n$. 
	
	We can similarly define cochain maps between two cochain complexes, homotopy between two cochain maps, homotopy equivalence between two cochains, and contractible cochains.
	
	Chain complexes of $\Lambda$-modules, together with chain maps between them, form a category, which we denote by $\Lambda\textendash\mathfrak{chain}$. Similarly, the category of cochain complexes of $\Lambda$-modules, together with cochain maps between them, form the category $\Lambda\textendash\mathfrak{cochain}$. When $\Lambda$ is the category of only one object and one morphism, we recover the categories $\mathfrak{chain}$ and $\mathfrak{cochain}$ of ordinary chain complexes and cochain complexes, respectively. 
\end{defn}

\begin{rmk}\label{rmk:chain-module}
	It is clear that a chain complex of right (resp., left) $\Lambda$-modules is the same as a contravariant (resp., covariant) functor from $\Lambda$ to $\mathfrak{chain}$. An analogous statement holds for cochain complexes. 
\end{rmk}

\subsection{Categorical \mbox{(co-)}homology}

In this subsection, we review the construction of \mbox{(co-)}homology for small categories.

\begin{defn}
	The \emph{homology} of a chain complex $(\mathcal{C}_n, \partial_n)$ of abelian groups (i.e., modules over the one-element category) is the sequence of abelian groups 
	\[H_n(\mathcal{C}_*) := \operatorname{Ker} \partial_n / \operatorname{Im} \partial_{n+1}.\]
	Dually, the \emph{cohomology} of a cochain complex $(\mathcal{C}^n, \delta_n)$ of abelian groups is the sequence of abelian groups $H^n(\mathcal{C}^*) := \operatorname{Ker} \delta_n / \operatorname{Im} \delta_{n-1}$. 
\end{defn}
Thus a \mbox{(co-)}chain complex is acyclic if and only if all of its \mbox{(co-)}homology groups vanish. Two \mbox{(co-)}chain complexes that are homotopy equivalent have the same \mbox{(co-)}homology groups.

\begin{defn}\label{def:categorical-free-res}
	A \emph{resolution}\footnote{What we call a resolution is in fact a left resolution. Since we will not need the dual notion of right resolutions in this paper, we drop the adjective ``left''. } of a (left/right) $\Lambda$-module $\mathcal{M}$ is an exact sequence 
	\[
		\ldots \overset{\partial_{n+1}}{\to} \mathcal{R}_{n} \overset{\partial_{n}}{\to} \mathcal{R}_{n-1} \overset{\partial_{n-1}}{\to} \ldots \overset{\partial_{1}}{\to} \mathcal{R}_{0} \overset{\partial_{0}}{\to} \mathcal{M} \overset{0}{\to} 0
	\]
	of $\Lambda$-modules and $\Lambda$-module maps. It is \emph{projective} (respectively, \emph{free}) if $\mathcal R_n$ is projective  (respectively, free over a base) for all $n$. It is convenient to view a resolution of $\mathcal{M}$ as an acyclic chain complex $\chaincomplex{ \mathcal{R}_n, \partial_n }_{n\in \Z}$, where $\mathcal{R}_{-1} = \mathcal{M}$ and $\mathcal{R}_{n} = 0$ for $n < -1$. 
\end{defn}

Remark~\ref{rmk:free-module-projective} implies that any free resolution is also projective. 

\begin{rmk}\label{rmk:categorical-resolution} 
	There is a canonical way to produce a free resolution $\chaincomplex{ \Presolution_n, \partial_n^{\Presolution}}_{n\in \Z}$ for the constant right $\Lambda$-module $\Z^\Lambda$ (cf. \cite[Definition~2.3]{gillaspy-kumjian} or \cite[page~2567]{xu}). To be more precise, for any $w \in \operatorname{Obj} \Lambda$ and any $n \in \N$, we define:
	\begin{itemize}
		\item $\Lambda^{*(n+1)}w$ to be the collection of composable $(n+1)$-tuples $(\lambda_0, \ldots, \lambda_n)$ with $s(\lambda_n) = w$,
		\item $\mathcal P_n(w) := \Z \Lambda^{*(n+1)}w$, the free abelian group generated by $\Lambda^{*(n+1)}w$, 
		\item $\partial_n^{\mathcal P} \colon \mathcal P_n \to \mathcal P_{n-1}$, for $n>0$, by prescribing
		\[ 
		\left(\partial_n^{\mathcal P} \right)_v (\lambda_0, \ldots, \lambda_n) = (\lambda_1, \ldots, \lambda_n) + \sum_{i=1}^n (-1)^i (\lambda_0, \ldots, \lambda_{i-1} \lambda_i, \ldots, \lambda_n) 
		\]
		on generators $(\lambda_0, \ldots, \lambda_n) \in \Lambda^{*(n+1)}v$, 
		\item and $\partial_0^{\mathcal P} : \Z \Lambda w \to \Z^\Lambda(w) = \Z$ by $\partial_0^{\mathcal P}(\lambda) = 1$. 
	\end{itemize}	
	For each $n \in \N$, right multiplication makes $\mathcal P_n$ into a right $\Lambda$-module:
	\[ 
		\mathcal P_n(\lambda)( \lambda_0, \ldots, \lambda_n) := (\lambda_0, \ldots, \lambda_{n-1}, \lambda_n \lambda) \text{ for any } ( \lambda_0, \ldots, \lambda_n) \in \Lambda^{*(n+1)}r(\lambda).
	\]
The fact that the $\Lambda$-module structure ignores the first $n$ components of each generating tuple motivated our choice of enumeration of the modules $\Presolution_n$: these first $n$ components carry the homological information in $\mathcal P_n$, whereas the last component only carries the $\Lambda$-module action.
	
	Moreover, using Definition~\ref{def:Z-lambda-star}, 
	we have a $\Lambda$-module isomorphism
	\[
		\mathcal P_n \cong \bigoplus_{v \in \operatorname{Obj}\Lambda} \; \bigoplus_{( \lambda_0, \ldots, \lambda_{n-1})  \in \Lambda^{*n}v } \Z  v  \Lambda  *
	\]
	given on the generators by sending $(\lambda_0, \ldots, \lambda_n) $ to the element $\lambda_n $ in the copy of $\Z v \Lambda * $ 
	 indexed by $(\lambda_0, \ldots, \lambda_{n-1}) \in \Lambda^{* n}$.  
	Thus, it follows from Remark~\ref{rmk:free-module-Yoneda} that $\mathcal P_n$ is free over $\bigsqcup_{v \in \operatorname{Obj}\Lambda}  \Lambda^{*n}v $.  
	Acyclicity was established in \cite[Proposition 2.4]{gillaspy-kumjian};  the chain homotopy $\chaincomplex{h_n \colon \Presolution_n \to \Presolution_{n+1} }_{n\in \Z}$ satisfying $\operatorname{id}_{\Presolution_*} = \partial_{n+1}^{\Presolution} h_n + h_{n-1} \partial_n^{\Presolution}$ is given by 
	\[
		\left(h_{-1}\right)_w ( 1 ) =  w \in \Z \Lambda^{*1} w
	\]
	and
	\[
		\left(h_{n}\right)_w ( \lambda_0, \ldots, \lambda_n) = (-1)^{n+1} ( \lambda_0, \ldots, \lambda_n, w)
	\]
	for any $n \in \N$, $w \in \operatorname{Obj}\Lambda$ and $(\lambda_0, \ldots, \lambda_n) \in \Lambda^{*(n+1)}w$. 
	Combining all these, we conclude that $\chaincomplex{ \Presolution_n, \partial_n^{\Presolution}}_{n\in \Z}$ is a free resolution for the right $\Lambda$-module $\Z^\Lambda$. We may also write $\mathcal P_n(\Lambda)$ for $\mathcal P_n$ to emphasize the $k$-graph $\Lambda$. 
\end{rmk}

Any two projective resolutions of $\Z^\Lambda$ are homotopy equivalent (cf.~\cite[Theorem~III.6.1]{maclane}), making the following definition independent of the choice of a projective resolution. 

\begin{defn}\label{def:categorical-cohlogy}
	Let $\Lambda$ be a small category, let $\mathcal{M}$ be a right $\Lambda$-module and $\mathcal{N}$ a left $\Lambda$-module. Pick a projective resolution 
	\[
		\ldots \overset{\partial_{n+1}}{\to} \mathcal{R}_{n} \overset{\partial_{n}}{\to} \mathcal{R}_{n-1} \overset{\partial_{n-1}}{\to} \ldots \overset{\partial_{2}}{\to} \mathcal{R}_{1} \overset{\partial_{1}}{\to} \mathcal{R}_{0} \overset{\partial_{0}}{\to} \Z^\Lambda \overset{0}{\to} 0
	\]
	for the constant right $\Lambda$-module $\Z^\Lambda$. Then
	\begin{itemize}
		\item the \emph{cohomology} $H^*(\Lambda, \M)$ of $\Lambda$ with coefficient $\mathcal{M}$ is defined to be the cohomology of the cochain complex 
		\[
			\ldots \overset{\partial_{n+1}^*}{\leftarrow} \Hom_\Lambda(\mathcal R_{n}, \M) \overset{\partial_{n}^*}{\leftarrow}
			\ldots \overset{\partial_{2}^*}{\leftarrow} \Hom_\Lambda(\mathcal R_1, \M) \overset{\partial_{1}^*}{\leftarrow} \Hom_\Lambda(\mathcal R_0, \M) \overset{0}{\leftarrow} 0 \overset{0}{\leftarrow} \ldots
		\]
		where $\partial_{n+1}^*(\eta) = \eta \circ \partial_{n+1}$ for any $\eta \in \Hom_\Lambda(\mathcal R_n, \M)$, and 
		\item the \emph{homology} $H_*(\Lambda, \mathcal N)$ of $\Lambda$ with coefficient $\mathcal{N}$ is defined to be the homology of the chain complex 
		\[
			\ldots \overset{\partial_{n+1} \otimes \operatorname{id}}{\to} \mathcal{R}_n \otimes_\Lambda \mathcal{N} \overset{\partial_{n} \otimes \operatorname{id}}{\to}  \ldots \overset{\partial_{2} \otimes \operatorname{id}}{\to} \mathcal{R}_1 \otimes_\Lambda \mathcal{N} \overset{\partial_{1} \otimes \operatorname{id}}{\to} \mathcal{R}_0 \otimes_\Lambda \mathcal{N} \overset{0}{\to} 0 \overset{0}{\to} \ldots
		\]
		where $(\partial_n \otimes \operatorname{id} ) (m \otimes k) = \partial_n (m) \otimes k$ for any elementary tensor $m \otimes k$. 
	\end{itemize}
	Notice that the term $\mathbb{Z}^{\Lambda}$ in the projective resolution is dropped when calculating the \mbox{(co-)}homology. If this term is included instead, we obtain \emph{reduced \mbox{(co-)}homology}. 
	
	In the case where $\Lambda$ is a $k$-graph, in order to distinguish the above definitions from the cubical homology and cohomology (cf.~Section~\ref{sec:cubical}), we refer to the above as the \emph{categorical \mbox{(co-)}homology} of $\Lambda$.  
\end{defn}

\begin{rmk}\label{rmk:categorical-cohlogy-reformulations}
	In Proposition 2.8 of \cite{gillaspy-kumjian}, the categorical cohomology is reformulated as the cohomology group of the so-called \emph{categorical $n$-cochains}. An analogous reformulation works for homology, too. We indicate how to see these identifications 
	through the perspective of free modules as in Definition~\ref{def:free-module}. Recall from Remark~\ref{rmk:categorical-resolution}  that  the modules  $\mathcal P_n$ are free over $(B_v)_{v\in \operatorname{Obj}\Lambda}$ with
	\[ B_v = \left\{ ( \lambda_0, \ldots, \lambda_{n-1}) \in \Lambda^{*n}v \right\}.\]
	Consequently, for any $n \in \N$, any right $\Lambda$-module $\mathcal{M}$ and any left $\Lambda$-module $\mathcal{N}$, Remark~\ref{rmk:free-module-isomorphisms} gives isomorphisms
	\begin{align*}
		\operatorname{Hom}_\Lambda \left(\mathcal P_n, \mathcal{M} \right) & \overset{\cong}{\to} \prod_{( \lambda_0, \ldots, \lambda_{n-1}) \in \Lambda^{*n}} \mathcal{M}(s(\lambda_{n-1})) \\
		f & \mapsto \left( f ( \lambda_0, \ldots, \lambda_{n-1})  \right)_{( \lambda_0, \ldots, \lambda_{n-1}) \in \Lambda^{*n}} 
	\end{align*}
	and
	\begin{align*}
		\bigoplus_{( \lambda_0, \ldots, \lambda_{n-1}) \in \Lambda^{*n}} \mathcal{N}(s(\lambda_{n-1})) & \overset{\cong}{\to}  \mathcal P_n \otimes_\Lambda \mathcal{N} \\  
		\left( a_{( \lambda_0, \ldots, \lambda_{n-1})} \right)_{\Lambda^{*n}} & \mapsto \sum_{( \lambda_0, \ldots, \lambda_{n-1}) \in \Lambda^{*n}} ( \lambda_0, \ldots, \lambda_{n-1}) \otimes a_{( \lambda_0, \ldots, \lambda_{n-1})} \; .
	\end{align*}
	We thus define
	\[
		C^n ( \Lambda, \mathcal{M}) := \prod_{( \lambda_0, \ldots, \lambda_{n-1}) \in \Lambda^{*n}} \mathcal{M}(s(\lambda_{n-1})) \; ,
	\]
	the group of \emph{$\mathcal{M}$-valued categorical $n$-cochains}, and 
	\[
		C_n ( \Lambda, \mathcal{N}) := \bigoplus_{( \lambda_0, \ldots, \lambda_{n-1}) \in \Lambda^{*n}} \mathcal{N}(s(\lambda_{n-1})) \; ,
	\]
	the group of \emph{$\mathcal{N}$-valued categorical $n$-chains}. Under these identifications, the coboundary maps $\delta_{n}$ for cochains satisfy
	\begin{align*}
		\delta_{n} ( f) (\lambda_0, \ldots, \lambda_{n})  = &\ f(\lambda_1, \ldots, \lambda_{n}) + (-1)^{n+1} f(\lambda_0, \ldots, \lambda_{n-1}) \cdot \lambda_{n} \\
		& + \sum_{i=1}^{n} (-1)^i f(\lambda_0, \ldots, \lambda_{i-1} \lambda_i, \ldots, \lambda_{n}) 
	\end{align*}
	for any $( \lambda_0, \ldots, \lambda_{n}) \in \Lambda^{*(n+1)}$ and $f \in C^n ( \Lambda, \mathcal{M})$, viewed as a section, where we used the right multiplication by $\lambda_{n}$ given by the module structure. Writing a typical element in $C_n ( \Lambda, \mathcal{N})$ as a pair $\big(  {( \lambda_0, \ldots, \lambda_{n-1}), a } \big)$ where $( \lambda_0, \ldots, \lambda_{n-1}) \in \Lambda^{*n}$ and $a \in \mathcal{N}({s(\lambda_{n-1})} )$, the boundary maps $\partial_{n}$ for chains satisfy
	\begin{align*}
		\partial_{n} \big(  {( \lambda_0, \ldots, \lambda_{n-1}), a } \big) = &\ \big({(\lambda_1, \ldots, \lambda_{n-1})}, a \big) + (-1)^n  \big({(\lambda_0, \ldots, \lambda_{n-2})}, \lambda_{n-1} \cdot  a\big)  \\
		& + \sum_{i=1}^{n-1} (-1)^i \big(  {(\lambda_0,\ldots, \lambda_{i-1} \lambda_i, \ldots, \lambda_{n-1}), a } \big) \;,
	\end{align*}
	where we used the left multiplication by $\lambda_{n-1}$ given by the module structure.  
	With these definitions, the categorical \mbox{(co-)}homology becomes just the \mbox{(co-)}homology of the categorical \mbox{(co-)}chain groups together with the above \mbox{(co-)}boundary maps.  
	
	In particular, for a $k$-graph $\Lambda$ and an abelian group $M$, the cohomology groups $H^n(\Lambda, M^{\Lambda})$ as in Definition~\ref{def:categorical-cohlogy} agree with the categorical cohomology groups of $\Lambda$ as defined in \cite[Definition 3.5]{kps-twisted}.   Moreover, the  categorical \mbox{(co-)}homology of a $k$-graph enjoys stronger functoriality properties than one might expect: 
	any functor $\varphi: \Lambda \to \Lambda'$ between two $k$-graphs will induce homomorphisms
on the categorical \mbox{(co-)}homology groups, regardless of whether $\varphi$ is a $k$-graph morphism (i.e., respects 
the degree functors).
\end{rmk}

\section{Cubical \mbox{(co-)}homology with coefficients}\label{sec:cubical}
In this section, we review the treatment of cubical \mbox{(co-)}homology for $k$-graphs, which was introduced by Kumjian, Pask, and Sims in  \cite{kps-hlogy}, before going on to explain how to incorporate general $\Lambda$-modules as coefficients. The main motivation for introducing cubical homology and cohomology is their ease for computation, as compared to categorical \mbox{(co-)}homology. 

Throughout this subsection, $\Lambda$ denotes a $k$-graph. 
 
\begin{defn}\label{def:cubical}
	As in \cite[Sections 2 and 3]{kps-hlogy} and using the notations in Definition~\ref{defn:k-graph}, for $n \in \N$, we define the set of \emph{$n$-cubes} in a $k$-graph $\Lambda$ by
	\[Q_n(\Lambda) := \left\{ \eta \in \Lambda\colon \sum_{i=1}^k d_i(\eta) = n \text{ and } d_i(\eta) \in \{0,1\} \text{ for } 1 \leq i \leq k \right\}.\]
	Observe that $Q_0(\Lambda) = \Lambda^0$. 
	For $n \in \Z \setminus \N$, we set $Q_n(\Lambda) = \varnothing$. We also write $Q = \bigcup_{n \in \Z} Q_n$, the set of all cubes.
	
	For $n \in \mathbb{N}$ and $\lambda \in Q_n(\Lambda)$ with $d(\lambda) = \sum_{j=1}^n e_{c_j}$ so that $1 \leq c_1 < \cdots < c_n \leq k$, using the notations in Remark~\ref{rmk:Lambda-m-n}, we define the \emph{front and back faces}
	\begin{equation}F^0_j(\lambda) = \lambda(0, d(\lambda) - e_{c_j}), \quad F^1_j(\lambda) = \lambda(e_{c_j}, d(\lambda)).\label{eq:cubical-faces}
	\end{equation}
	In other words, $\lambda$ can be decomposed both as $F^0_j(\lambda) S_j(\lambda)$ and as $R_j(\lambda) F^1_j(\lambda)$ such that $d(S_j(\lambda)) = d(R_j(\lambda)) = e_{c_j}$. In fact, $S_j(\lambda) = \lambda(d(\lambda) - e_{c_j}, d(\lambda))$ and $R_j(\lambda) = \lambda(0,  e_{c_j})$. Note that in \cite{kps-hlogy}, $R_j(\lambda)$ and $S_j(\lambda)$ are denoted by $\alpha_j$ and $\beta_j$, respectively. 
\end{defn}

\newcommand{\ZQ}{{\Z Q}}
\newcommand{\ZtildeQ}{{\Z \widetilde{Q}}}

\begin{defn}\label{def:cubical-chain-complex}
The \emph{cubical chain complex} $ {\ZQ}_*(\Lambda)$ consists of the free abelian groups $\chaincomplex{\Z Q_n(\Lambda)}_{n\in \Z}$ and boundary maps $\partial_n\colon \ZQ_n(\Lambda) \to \ZQ_{n-1}(\Lambda)$ defined, for $n \geq 1$, by
\begin{equation} 
\partial_n(\lambda) = \sum_{j=1}^n \sum_{l=0}^1 (-1)^{j+l} F^l_j(\lambda),
\label{eq:cubical-bdry}
\end{equation}
for any $n$-cube $\lambda$.

The \emph{augmented cubical chain complex} ${\ZtildeQ}_*(\Lambda)$ is defined in the same way as above, with
\[
	\ZtildeQ_n (\Lambda) = \begin{cases}
		\ZQ_n(\Lambda) \, , & n \not= -1 \\
		\Z \, , & n = -1
	\end{cases} \; ,
\]
and the boundary map $\partial_0 \colon  \mathbb{Z} Q_0(\Lambda) \to \mathbb{Z}$ maps any $v \in \Lambda^0$ to $1$ in $\mathbb{Z}$. 

It is straightforward to verify that $\partial_n \circ \partial_{n+1} = 0$ for all $n$ for either chain complex. Thus, the \emph{cubical homology group} of $\Lambda$, written $H_{n}^{\operatorname{cub}}(\Lambda)$, is defined as the homology of $\chaincomplex{\Z Q_n(\Lambda), \partial_n }_n$, while the \emph{reduced cubical homology group} of $\Lambda$, written $\widetilde{H}_{n}^{\operatorname{cub}}(\Lambda)$, is defined as the homology of $\chaincomplex{\Z \widetilde Q_n(\Lambda), \partial_n}_n$.
\end{defn}

\begin{rmk}\label{rmk:reduced-homology}
	It is easy to see that for any nontrivial $k$-graph, the (unreduced) cubical homology differs from the reduced cubical homology only by adding a copy of $\mathbb{Z}$ as a direct summand to the $0$-th homology group: 
	\[H_{0}^{\operatorname{cub}}(\Lambda) \cong \widetilde{H}_{0}^{\operatorname{cub}}(\Lambda) \oplus \mathbb{Z}\quad \text{ and } \quad H_{n}^{\operatorname{cub}}(\Lambda) \cong \widetilde{H}_{n}^{\operatorname{cub}}(\Lambda) \ \forall \ n \not= 0.\]
\end{rmk}

\begin{rmk}\label{rmk:ZQ-functoriality}
	We point out that the assignment $\Lambda \mapsto \chaincomplex{ \Z Q_n(\Lambda) }_n$ (resp., $\chaincomplex{\Z \widetilde{Q}_n(\Lambda) }_n$) constitutes a covariant functor $\Z Q_*$ (resp., $\Z \widetilde{Q}_*$) from the category $k \text{\textendash} \mathfrak{graph}$ to the category $\mathfrak{chain}$, since a $k$-graph morphism $\varphi \colon  \Lambda \to \Lambda'$ preserves degree and thus induces maps $(\varphi_*)_{n} \colon \Z Q_n(\Lambda) \to \Z Q_n(\Lambda')$ that intertwine the boundary maps $\partial_n$. 
\end{rmk}

More generally, mimicking the definitions in Remark~\ref{rmk:categorical-cohlogy-reformulations}, we can incorporate a coefficient left $\Lambda$-module into the definition of the cubical homology.

\begin{constr}\label{constr:cubical-homology-pre}
	Let $\Lambda$ be a $k$-graph and let $\mathcal{N}$ be a left $\Lambda$-module. We define a chain complex $\chaincomplex{C_n^{\operatorname{cub}} ( \Lambda, \mathcal{N}), \partial_n \otimes \operatorname{id}}_n$ by setting, for each $n \in \N$, 
	\[
		C_n^{\operatorname{cub}} ( \Lambda, \mathcal{N}) = \bigoplus_{\lambda \in Q_n (\Lambda)} \mathcal{N}({s(\lambda)} )
	\]
	i.e., the linear span of pairs $(\lambda , a)$ for $\lambda \in Q_n(\Lambda)$ and $a \in \mathcal{N}({s(\lambda)} )$, and defining the differential map
	\[
		\partial_{n} \otimes \operatorname{id} \colon  C_{n}^{\operatorname{cub}} ( \Lambda, \mathcal{N}) \to C_{n-1}^{\operatorname{cub}} ( \Lambda, \mathcal{N}) \; 
	\] 
	such that for $\lambda \in Q_n(\Lambda)$ and $a \in \mathcal{N}({s(\lambda)} )$,
	\[
		\partial_{n} \otimes \operatorname{id} (\lambda , a) = \sum_{i=1}^n (-1)^{i} \left( (F_i^0(\lambda) ,S_i(\lambda) \cdot a)   -  (F_i^1(\lambda),a) \right) ,
	\]
	where $S_i(\lambda) \cdot a \in \mathcal N(r( S_i(\lambda)))$ arises from the left action of $\Lambda$ on $\mathcal N$.
\end{constr}
One readily verifies that $(\partial_n \otimes \operatorname{id}) \circ (\partial_{n+1} \otimes \operatorname{id}) = 0$ for all $n$. 

\begin{defn}\label{def:cubical-homology}
	Let $\Lambda$ be a $k$-graph and let $\mathcal{N}$ be a left $\Lambda$-module. The {\em cubical homology $H_n^{\operatorname{cub}}(\Lambda,\mathcal N)$ of $\Lambda$ with coefficients in} $\mathcal{N}$  is defined to be the homology of the chain complex $\chaincomplex{C_n^{\operatorname{cub}} ( \Lambda, \mathcal{N}), \partial_n\otimes \operatorname{id} }_{n\in \N}$. 
\end{defn}

Similarly, we can define the cubical cohomology with a right $\Lambda$-module as its coefficient. 
\begin{constr}
\label{constr:cubical-cohomology-pre}
	Let $\Lambda$ be a $k$-graph and let $\mathcal{M}$ be a right $\Lambda$-module. 
	We define a cochain complex $\chaincomplex{ C^n_{\operatorname{cub}} ( \Lambda, \mathcal{M}), \delta_n }_n$ by setting, for each $n \in \N$, 
	\[
		C^n_{\operatorname{cub}} ( \Lambda, \mathcal{M}) = \prod_{\lambda \in Q_n (\Lambda)} \mathcal{M}({s(\lambda)} )
	\]
	and defining the differential map
	$
		\delta_{n-1}  \colon  C^{n-1}_{\operatorname{cub}} ( \Lambda, \mathcal{M}) \to C^{n}_{\operatorname{cub}} ( \Lambda, \mathcal{M}) \; 
	$
by
\[
 \delta_{n-1} ( f ) (\lambda) = \sum_{i=1}^n (-1)^{i} \left(f ( F_i^0(\lambda) ) \cdot S_i(\lambda) -   f ( F_i^1(\lambda) )\right) 
\]
for any $\lambda \in Q_n{(\Lambda)}$ and $f \in C^{n-1}_{\operatorname{cub}} ( \Lambda, \mathcal{M})$, considered as a tuple of elements $f(\mu) \in \mathcal{M}({s(\mu)})$ for varying $\mu \in Q_{n-1}(\Lambda)$. Here we used the right multiplication by $S_i(\lambda)$ as prescribed by the right module structure.
\end{constr}
One readily verifies that $\delta_{n+1} \circ \delta_{n} = 0$ for all $n$. 

\begin{defn}\label{def:cubical-cohomology}
 Let $\Lambda$ be a $k$-graph and let $\mathcal{M}$ be a right $\Lambda$-module. The {\em cubical cohomology $H^n_{\operatorname{cub}}(\Lambda,\mathcal M)$ of $\Lambda$ with coefficients in} $\mathcal{M}$  is defined to be the cohomology of the cochain complex $\chaincomplex{C^n_{\operatorname{cub}} ( \Lambda, \mathcal{M}), \delta_n }_{n\in \N}$.
\end{defn}

\begin{rmk}\label{rmk:cubical-vanish}
	Observe that for a $k$-graph $\Lambda$, the chain complex
	\[ \chaincomplex{C_n^{\operatorname{cub}} ( \Lambda, \mathcal{N}), \partial_n \otimes \operatorname{id}}_{n\in \N}\]
	 and the cochain complex $\chaincomplex{C^n_{\operatorname{cub}} ( \Lambda, \mathcal{M}), \delta_n }_{n\in \N}$ both vanish for $n>k$, regardless of the coefficient modules. Consequently, $H_n^{\operatorname{cub}}(\Lambda,\mathcal N)$ and $H^n_{\operatorname{cub}}(\Lambda,\mathcal M)$ always vanish when $n>k$, as well as when $n<0$. 
\end{rmk}

\begin{rmk}
	\label{rmk:cubical-coh-constant-module}
	When $\mathcal{M} = M^\Lambda$ is the constant module generated by an abelian group $M$,  the cochain complex $\chaincomplex{C^n_{\operatorname{cub}} ( \Lambda, \mathcal{M}), \delta_n }_{n\in \N}$ is isomorphic to $\Hom(\Z Q_n(\Lambda) , M)$. This implies that Definition \ref{def:cubical-cohomology} agrees with the definition of cubical cohomology given in  \cite[Definition 7.2]{kps-hlogy}. 
\end{rmk}

\section{The cubical free resolution}
\label{sec:new-free-res}
In this section, we construct the cubical free resolution of a $k$-graph, which plays a central role in both of our proofs showing the isomorphism between cubical \mbox{(co-)}homology groups and categorical \mbox{(co-)}homology groups. 
\begin{defn}\label{def:future-path-graph}
	Let $\Lambda$ be a $k$-graph. For any vertex $v \in \Lambda^0$, we define $\widetilde{\Lambda} (v)$, \emph{the future path $k$-graph of $\Lambda$ rooted at $v$} as follows:
	\begin{itemize}
		\item As a category, $\widetilde{\Lambda} (v)$ is the coslice category of $\Lambda$ at $v$. In other words, this is the small category whose set of objects is $\Lambda v$ (the morphisms in $\Lambda$ originating from $v$), and the set of morphisms is $\Lambda^{*2} v$ (the composable 2-tuples in $\Lambda$ which originate from $v$).  For any morphism $(\eta, \mu) \in \Lambda^{*2} v$, its source $\tilde{s}(\eta, \mu)$ is $\mu$ and its range $\tilde r(\eta , \mu)$ is $\eta \mu$. 
		\item The degree of a morphism $(\eta, \mu)$, written $\tilde{d}(\eta, \mu )$, is given by $d(\eta)$. 
	\end{itemize}
\end{defn}

\begin{rmk}\label{rmk:future-path-graph} The $k$-graphs $\widetilde{\Lambda} (v)$ have the following properties. 
	\begin{enumerate}
		\item\label{rmk:future-path-graph:forgetful} There is a forgetful functor $F \colon  \widetilde{\Lambda} (v) \to \Lambda$, given by $F(\eta, \mu ) = \eta$. The factorization property in $\Lambda$ ensures $F$ is faithful. It also preserves the degree of morphisms since  $\tilde{d} = d \circ F$. 
		\item\label{rmk:future-path-graph:degree} The degree map $\tilde{d}$ satisfies the factorization property and thus $\widetilde{\Lambda} (v)$ is a bona fide $k$-graph. Indeed, suppose $(\eta, \lambda) \in \widetilde{\Lambda}(v)$ has $\widetilde{d}(\eta, \lambda) = d(\eta) = m +n$.  The factorization property for $\Lambda$ implies that $\eta$ factors uniquely as $\eta = \mu \nu$ where $d(\mu)= m, d(\nu) = n$.  Thus, 
		\[ (\eta, \lambda) = (\mu, \nu \lambda)(\nu, \lambda)\]
		gives the factorization in $\widetilde{\Lambda}(v)$, unique by the faithfulness of $F$. 
		\item\label{rmk:future-path-graph:initial} The vertex $v$ is \emph{initial} in $\widetilde{\Lambda} (v)$, that is, for any vertex $\lambda $ in $\widetilde{\Lambda} (v)$, there is only one morphism \textemdash\  namely $(\lambda, v)$ \textemdash\  from $v$ to $\lambda$.
		\item\label{rmk:future-path-graph:functor-1} Each morphism $\lambda$ in $\Lambda$ induces a functor
		\[
		\widetilde{\Lambda} (\lambda) \colon  \widetilde{\Lambda} (r(\lambda)) \to \widetilde{\Lambda} (s(\lambda)) \;, \quad (\eta, \mu) \mapsto (\eta, \mu \lambda) \; .
		\]
		It is a $k$-graph morphism by the following commutative diagram
		\[
		\xymatrix{
			\widetilde{\Lambda} (r(\lambda)) \ar[rr]^{\widetilde{\Lambda}(\lambda)} \ar[rd]^F &&  \widetilde{\Lambda} (s(\lambda)) \ar[ld]_F \\
			&\Lambda \ar[r]^d & \mathbb{N}
		}
		\]
		Thus $\widetilde{\Lambda}$ is a contravariant functor from $\Lambda$ to  $k \text{\textendash} \mathfrak{graph}$. 
		\item\label{rmk:future-path-graph:functor-2} A $k$-graph morphism $\varphi \colon  \Lambda \to \Lambda'$ induces a functor 
		\[
		\widetilde{\varphi}(v) \colon  \widetilde{\Lambda}(v) \to \widetilde{\Lambda'} (\varphi(v)) \;, \quad (\eta, \mu) \mapsto (\varphi(\eta), \varphi(\mu)) \; ,
		\]
		which is a $k$-graph morphism by the following commutative diagram:
		\[
		\xymatrix{
			\widetilde{\Lambda} (v) \ar[rr]^{\widetilde{\varphi}(v)} \ar[d]^F &&  \widetilde{\Lambda'} (\varphi(v)) \ar[d]_F \\
			{\Lambda} \ar[rr]^{\varphi} \ar[rd]^d &&  {\Lambda'}  \ar[ld]_d \\
			& \mathbb{N}
		}
		\]
		\item\label{rmk:future-path-graph:natural-morphism} In fact, the above functors form a natural transformation \[
		\widetilde{\varphi} \colon  \widetilde{\Lambda} \to \widetilde{\Lambda'}  \varphi
		\]
		between functors from $\Lambda$ to $k \text{\textendash} \mathfrak{graph}$, thanks to the following commutative diagram (for each $\lambda \in \Lambda$):
		\[
		\xymatrix{
			\widetilde{\Lambda} (r(\lambda)) \ar[rr]^{\widetilde{\varphi}(r(\lambda))} \ar[d]^{\widetilde{\Lambda} (\lambda)} &&  \widetilde{\Lambda'} (\varphi(r(\lambda))) \ar[d]^{\widetilde{\Lambda'} (\varphi(\lambda))} \\
			\widetilde{\Lambda} (s(\lambda)) \ar[rr]^{\widetilde{\varphi}(s(\lambda))}  &&  \widetilde{\Lambda'} (\varphi(s(\lambda))) 
		}
		\]
	\end{enumerate}
\end{rmk}

\begin{example}
	If $\Lambda$ is the $1$-graph associated to the figure-8 directed graph, made up of one single vertex $v$ and two edges, then the future path $1$-graph $\widetilde{\Lambda}(v)$ can be identified with the rooted binary tree, while the functor $\widetilde{\Lambda}$ maps any path in $\Lambda$ to a graph endomorphism of the rooted binary tree. In general, a future path $1$-graph is always a rooted tree. 
\end{example}

\begin{example}
	If $\Lambda = \mathbb{N}^k$ as a $k$-graph (with the degree map being the identity), then the future path $k$-graph $\widetilde{\Lambda}(0)$ is $\Omega^k$, i.e., the poset category $(\mathbb{N}^k, \leq)$ viewed as a $k$-graph. More generally, a future path $k$-graph can be thought of as a branched version of $\Omega^k$, just as a rooted tree can be thought of as a branched version of $\Omega^1$. 
\end{example}

Next we consider cubical complexes over future path $k$-graphs. 

\begin{constr}\label{constr:cubical-resolution}
	Recall from Remark~\ref{rmk:ZQ-functoriality} that $\Z Q_*$ and $\Z \widetilde{Q}_*$ are covariant functors from the category $k \text{\textendash} \mathfrak{graph}$ of $k$-graphs to the category $\mathfrak{chain}$ of chain complexes (of abelian groups). Thus composing them with the contravariant functor $\widetilde{\Lambda} \colon  \Lambda \to k \text{\textendash} \mathfrak{graph}$ produces contravariant functors
	\[
	\Z Q_*(\widetilde{\Lambda}) \text{ and } \Z \widetilde{Q}_*(\widetilde{\Lambda})  \colon  \Lambda \to \mathfrak{chain} \; ,
	\]
	which, by Remark~\ref{rmk:chain-module}, can be viewed as chain complexes of right $\Lambda$-modules. 
	The chain complex $\chaincomplex{\Z \widetilde{Q}_n(\widetilde{\Lambda}) }_n$ is the same as $\chaincomplex{\Z Q_n(\widetilde{\Lambda}) }_n$ except at $n = -1$, where instead of being trivial, we have $\Z \widetilde{Q}_{-1}(\widetilde{\Lambda}) = \Z^\Lambda$, the $\Z$-valued constant right $\Lambda$-module. 
	\end{constr}

As in the case of the categorical free resolution (Remark \ref{rmk:categorical-resolution}), in any generating tuple $(\eta, \lambda)$ of $\Z Q_n(\widetilde{\Lambda})(v)$, where $\eta \in Q_n(\Lambda)$ and $\lambda \in  s(\eta) \Lambda v$, only the second and final entry $\lambda$ is affected by the right $\Lambda$-module structure of $\Z Q_n(\widetilde{\Lambda})$.  Intuitively speaking, the first entry $\eta$ carries the (cubical) homological information.

\begin{prop}\label{prop:cubical-module-free}
	For any $n \in \Z$, the right $\Lambda$-module $\Z Q_n(\widetilde{\Lambda})$ is 
	free over the base $\left\{\big(\eta, s(\eta)\big)\colon \eta \in Q_n(\Lambda) \right\}$ in the sense of Definition~\ref{def:free-module}. 
\end{prop}

\begin{proof}
	Fix $n \in \N$. For any $v \in \Lambda^0$,  Construction~\ref{constr:cubical-resolution} implies that the abelian group $\Z Q_n(\widetilde{\Lambda}) (v)$ is freely generated by
	\[ 
	\left\{ (\eta, \lambda) \colon  \eta \in Q_n(\Lambda), \lambda \in  s(\eta) \Lambda v \right\} \; .
	\]
	Hence the $\Lambda$-module map 
	\begin{align*}
		\bigoplus_{\eta \in Q_n(\Lambda)}  \Z \: s(\eta){\Lambda}  \ast &\to \Z Q_n(\widetilde{\Lambda}) \\
		(\lambda_\eta)_{\eta\in Q_n(\Lambda)} &\mapsto \sum_{\eta \in Q_n(\Lambda)} (\eta, \lambda_\eta)
	\end{align*}
	is an isomorphism. The proposition follows by Remark~\ref{rmk:free-module-Yoneda}. 
\end{proof}

It follows that $\Z Q_n(\widetilde{\Lambda})$ is closely related to the cubical \mbox{(co-)}homology groups, which are defined respectively through the cubical chain complex $C_*^{\operatorname{cub}} ( \Lambda, \mathcal{N})$ in Construction~\ref{constr:cubical-homology-pre}  and the cubical cochain complex $C^*_{\operatorname{cub}} ( \Lambda, \mathcal{M})$ in Construction~\ref{constr:cubical-cohomology-pre}. 

\begin{prop}\label{prop:cubical-chain-iso}
	For any $k$-graph $\Lambda$, any left $\Lambda$-module $\mathcal{N}$ and right $\Lambda$-module $\mathcal{M}$, there are isomorphisms
	\[
	C_*^{\operatorname{cub}} ( \Lambda, \mathcal{N}) \cong \Z Q_*(\widetilde{\Lambda})\otimes_\Lambda \mathcal{N}
	\]
	and
	\[
	C^*_{\operatorname{cub}} ( \Lambda, \mathcal{M}) \cong \Hom_\Lambda\left(\Z Q_*(\widetilde{\Lambda}), \mathcal{M}\right)
	\]
	of \mbox{(co-)}chain complexes. 
\end{prop}

\begin{proof}
	By Proposition~\ref{prop:cubical-module-free}, for each $n \in \Z$, the right $\Lambda$-module $\Z Q_n(\widetilde{\Lambda})$ is free over the base $\left\{ \big(\eta, s(\eta)\big)\colon \eta \in Q_n(\Lambda) \right\}$. Thus by Remark~\ref{rmk:free-module-isomorphisms}, there are group isomorphisms
	\[
	\Psi_n \colon \Hom_\Lambda\left(\Z Q_n(\widetilde{\Lambda}), \mathcal{M}\right) \overset{\cong}{\to} \prod_{\eta \in Q_n (\Lambda)} \mathcal{M}({s(\eta)} ) \, , \quad f \mapsto \left( f_{s(\eta)} (\eta, s(\eta)) \right)_{\eta \in Q_n (\Lambda)} 
	\]
	and 
	\[
	\Phi_n \colon \bigoplus_{\eta \in Q_n (\Lambda)} \mathcal{N}({s(\eta)} ) \overset{\cong}{\to}  \Z Q_n(\widetilde{\Lambda})\otimes_\Lambda \mathcal{N} \, , \quad  (a_\eta)_{\eta \in Q_n(\Lambda)} \mapsto \sum_{\eta \in Q_n (\Lambda)} (\eta, s(\eta)) \otimes a_\eta   \; .
	\]
	By definition, we have $C^n_{\operatorname{cub}} ( \Lambda, \mathcal{M}) = \prod_{\eta \in Q_n (\Lambda)} \mathcal{M}({s(\eta)} )$ and $C_n^{\operatorname{cub}} ( \Lambda, \mathcal{N}) = \bigoplus_{\eta \in Q_n (\Lambda)} \mathcal{N}({s(\eta)} )$. It remains to show that these isomorphisms intertwine the differential maps. To this end, we compute, for any $f \in \Hom_\Lambda(\Z Q_n( \widetilde{\Lambda}), \mathcal M)$ and $\eta \in Q_{n+1}$,
	\begin{align*}
	&\ \left(\delta_{n} \circ \Psi_n\right)(f)(\eta) \\
	&= \sum_{i=1}^n (-1)^i \left( \Psi_n(f) \left(F^0_i(\eta) \right) \cdot S_i(\eta) -  \Psi_n(f) \left(F^1_i(\eta) \right) \right)\\
	& = \sum_{i=1}^n (-1)^i \left( f_{s(F_i^0(\eta))} \big(F^0_i(\eta), s(F^0_i(\eta)) \big) \cdot S_i(\eta) - f_{s(\eta)} (F^1_i(\eta), s(\eta)) \right)\\
	& = \sum_{i=1}^n (-1)^i \Big( f_{s(\eta)} (F^0_i(\eta), S_i(\eta) ) - f_{s(\eta)} (F^1_i(\eta), s(\eta)) \Big)\\
	& = f_{s(\eta)} \left(\sum_{i=1}^n (-1)^i \Big( \left( F^0_i(\eta), S_i(\eta)\right)  - \left(F^1_i(\eta), s(\eta) \right)\Big)\right)\\
	& = \left(f \circ \partial_n\right)_{s(\eta)} (\eta, s(\eta)) = \Psi_{n+1} \left(f \circ \partial_n\right) (\eta) \\
	& = \left( \Psi_{n+1} \circ \partial_n^* \right)(f) (\eta) \; .
	\end{align*}
	This proves $\delta_{n} \circ \Psi_n = \Psi_{n+1} \circ \partial_n^*$ for all $n \in \Z$ and yields the desired cochain isomorphism.

	A similar computation establishes that 
	$\left(\partial_{n} \otimes \operatorname{id}_{\mathcal{N}} \right) \circ \Phi_n = \Phi_{n-1} \circ \partial_{n}$ for all $n \in \Z$.
\end{proof}

To justify the term ``cubical free resolution'', we will need to show that the chain complex 
\begin{equation}
\label{eq:cubical-free-res}
\dots \stackrel{\partial_{n+1}}{\to} \Z Q_{n+1}(\widetilde{\Lambda}) \stackrel{\partial_n}{\to} \Z Q_n(\widetilde{\Lambda}) \stackrel{\partial_{n-1}}{\to} \dots \to \Z Q_0(\widetilde{\Lambda}) \stackrel{\partial_0}{\to} \Z^\Lambda \to 0
\end{equation}
of $\Lambda$-modules is acyclic. This will occupy more than half of the paper. We will provide two proofs of this fact, one topological and the other algebraic.

\section{Initial vertices and contractibility}\label{sec:initial-vertices}

In this section, we show that if a $k$-graph has an {initial vertex} in the category-theoretical sense, then all of its reduced cubical homology groups vanish. This is a crucial step in our first proof of the isomorphism between cubical \mbox{(co-)}homology and categorical \mbox{(co-)}homology. 

\begin{defn}\label{def:initial-vertex}
	Let $\Lambda$ be a $k$-graph. A vertex $\alpha \in \operatorname{Obj} \Lambda$ is an \emph{initial vertex} if for any vertex $v \in \operatorname{Obj} \Lambda$, the set $v\Lambda \alpha$ contains a unique element, which we denote by $\alpha_v$. For any $\lambda \in \Lambda$, we also define $\alpha_\lambda$ to be $\alpha_{r(\lambda)}$. 
\end{defn}

\begin{example}
	A $1$-graph with an initial vertex is nothing but a rooted tree, where all the edges point away from the root.
\end{example}

\begin{example}
	As discussed in Remark~\ref{rmk:future-path-graph}\eqref{rmk:future-path-graph:initial}, the faithfulness of the forgetful functor $F: \widetilde \Lambda \to \Lambda$ implies  that for any future path $k$-graph $\widetilde{\Lambda}(v)$, $v$ is  an initial vertex. 
\end{example}

Let $\alpha$ be an initial vertex in a $k$-graph $\Lambda$.  We evidently have $\alpha_\alpha = \alpha$. Moreover, for any $\lambda \in \Lambda$, we have $\alpha_{\alpha_\lambda} = \alpha_\lambda$ and $\alpha_\lambda = \alpha_{r(\lambda)} = \lambda \alpha_{s(\lambda)}$. This last equation implies that if $\lambda, \mu \in v \Lambda w$ we must have 
\[ d(\lambda) = d(\mu) = d(\alpha_v) - d(\alpha_w)\]
and hence, by the factorization property, $\lambda = \mu$.

In order to  study the reduced cubical homology groups of $k$-graphs with initial vertices, we invoke the notion of the topological realization $|\Lambda|$ of a $k$-graph $\Lambda$, introduced in \cite{kalisz-kumjian-quigg-sims}. It is a topological space whose homology groups coincide with the cubical homology groups of $\Lambda$ by \cite[Theorem~6.3]{kps-hlogy}.

\begin{defn}\cite[Definition~3.2]{kalisz-kumjian-quigg-sims}
	\label{def:top-realiz}
	For each $n=(n_1, \ldots, n_k) \in \N^k$, we denote by $[0, n]$ the subset of $\R^k$ given by 
	\[ [0, n] = \{ x = (x_1, \ldots, x_k) \in \R^k\colon 0 \leq x_i \leq n_i\}.\]
	Then  $|\Lambda|$ is the quotient space
	\[ |\Lambda| = \left. \left( \bigsqcup_{\lambda \in \Lambda} \{ \lambda\} \times [0, d(\lambda)] \right) \right/ \sim\]
	under the equivalence relation
	\begin{align*}
	(\mu, x) \sim (\nu, y) \Longleftrightarrow \ &  \mu(\lfloor x \rfloor, \lceil x \rceil) = \nu(\lfloor y \rfloor, \lceil y \rceil) \text{ and } x - \lfloor x \rfloor = y - \lfloor y \rfloor 
	\end{align*}
	using the notation of Remark~\ref{rmk:Lambda-m-n}.  Here $\lfloor-\rfloor$ and $\lceil - \rceil$ denote, respectively, the floor and ceiling functions applied to every coordinate of a tuple. 
	We denote a typical element in $|\Lambda|$ by $[(\lambda, x)]$, where $\lambda\in \Lambda$ and $x \in [0, d(\lambda)]$. 
\end{defn}

In plain language, we associate to each morphism in $\Lambda$ a hyper-rectangle whose size  equals $d(\lambda)$, and if two morphisms $\mu, \nu$ overlap on a hyper-rectangle $\lambda$, we glue the hyper-rectangles for $\mu$ and $\nu$ together along $\lambda$. 

\begin{lemma}[\cite{kalisz-kumjian-quigg-sims} Lemma 3.3]
\label{lem:top-realiz-relation}
	For any $\lambda \in \Lambda$, $\mu \in \Lambda r(\lambda)$, $\nu \in s(\lambda)\Lambda$, and $x \in [0, d(\lambda)]$, we have $(\mu \lambda \nu, x + d(\mu) ) \sim (\lambda, x)$. 
\end{lemma}

\begin{prop}
	\label{prop:top-realiz-contr}
	If a $k$-graph $\Lambda$ has an initial vertex $\alpha$, then its topological realization $|{\Lambda}|$ is contractible.
\end{prop}

\begin{proof}	
	Consider the continuous map
	\begin{align*}
		\widetilde{H} \colon \bigsqcup_{\lambda \in \Lambda} \left( [0,1] \times \left\{ \lambda \right\} \times [0, d(\lambda)] \right) & \to |{\Lambda}|  \\
		(t, \lambda, x) & \mapsto \left[\left(\alpha_\lambda, x + t (d(\alpha_\lambda) - x) \right)\right] \; . 
	\end{align*}
	We claim that $\widetilde H$ descends to a continuous map $H: [0,1] \times |\Lambda| \to |\Lambda|$.
	To see that $ H$ is well defined, suppose $(\lambda, x) \sim (\mu, y)$, i.e., $\lambda(\lfloor x \rfloor, \lceil x \rceil) = \mu(\lfloor y \rfloor, \lceil y \rceil)$ and $x - \lfloor x \rfloor = y - \lfloor y \rfloor =: z $.  Defining $\eta = \lambda(\lfloor x \rfloor, \lceil x \rceil)$, 
	we claim that 
	\[
		\left(\alpha_\lambda, x + t (d(\alpha_\lambda) - x) \right) \sim \left(\alpha_\eta, z + t (d(\alpha_\eta) - z) \right) \sim \left(\alpha_\mu, y + t (d(\alpha_\mu) - y) \right).
	\]
	Indeed,  writing $\lambda' =  \lambda(0, \lfloor x \rfloor)$, we have $x = d(\lambda') + z$ and $\alpha_\lambda = \lambda' \alpha_\eta$.  It follows that $d(\alpha_\lambda) = d(\alpha_\eta) + \lfloor x \rfloor$ 
	and hence that 
	\[
		d(\alpha_\lambda) - x =d(\alpha_\lambda)  - \lfloor x \rfloor + (\lfloor x \rfloor - x) = d(\alpha_\eta) - z \; .
	\]
	Consequently, Lemma~\ref{lem:top-realiz-relation} implies that 
	\[
		\left(\alpha_\lambda, x + t (d(\alpha_\lambda) - x) \right) = \left(\lambda'\alpha_\eta, d(\lambda') + z + t (d(\alpha_\eta) - z) \right)  \sim \left(\alpha_\eta, z + t (d(\alpha_\eta) - z) \right) \; ,
	\]	
	which proves the first equivalence in the claim. The second is proved similarly. Taken together, this shows that if $(\lambda, x) \sim (\mu, y)$, then $\widetilde{H}(t, \lambda, x) = \widetilde{H}(t, \mu, y)$. Therefore $\widetilde{H}$ factors through a map $H \colon [0,1] \times |{\Lambda}| \to |{\Lambda}|$, whose continuity follows from that of $\widetilde{H}$. 
	
	We compute, for any $\lambda \in \Lambda$ and $x \in [ 0 , d(\lambda)]$, that 
	\[
		H \big(1, [( \lambda, x)] \big) = \left[\left(\alpha_\lambda, x + (d(\alpha_\lambda) - x) \right)\right] = \left[\left(\alpha_\lambda, d(\alpha_\lambda) \right)\right] = [(\alpha, 0)]
	\]
	by  Lemma~\ref{lem:top-realiz-relation} and the fact that $\alpha_\lambda = \alpha_\lambda \alpha$. On the other hand, 
	\[
		H \big(0, [( \lambda, x)] \big) = \left[\left(\alpha_\lambda, x \right)\right] = \left[\left(\lambda, x \right)\right] \; , 
	\]
	which follows from the equation $\alpha_\lambda = \lambda \alpha_{s(\lambda)}$. 
	Therefore $H$ constitutes a homotopy from the identity map to the constant map onto $\left[\left(\alpha, 0 \right)\right]$. 
\end{proof}

\begin{cor}\label{cor:top-realiz-acyclic}
	If a $k$-graph $\Lambda$ has an initial vertex, then its reduced cubical homology groups all vanish. 
\end{cor}

\begin{proof}
	By Theorem 6.3 of \cite{kps-hlogy}, the group $H_n^{\operatorname{cub}}({\Lambda})$ is isomorphic to the $n$th homology group of its topological realization $|{\Lambda}|$, for any $n \in \Z$. It follows from Remark~\ref{rmk:reduced-homology} that for any $n \in \Z$, the \emph{reduced} cubical homology group $\widetilde{H}_n^{\operatorname{cub}}({\Lambda})$ is isomorphic to the $n$th \emph{reduced} homology group of $|{\Lambda}|$, which is thus trivial since $|{\Lambda}|$ is contractible by Proposition~\ref{prop:top-realiz-contr}. 
\end{proof}

Although we will not use the following facts, we explore some further consequences of the existence of an initial vertex below. 

\begin{rmk}
	If a $k$-graph $\Lambda$ has an initial vertex $\alpha$, then the reduced categorical homology groups $\widetilde{H}_n(\Lambda)$ all vanish. This follows from \cite[Corollary~2]{Quillen}.  Indeed, the reduced categorical homology groups are computed by the augmented complex 
	\[
		\ldots \overset{\partial_{n+1}}{\to} \Z \Lambda^{*n} \overset{\partial_{n}}{\to} \Z \Lambda^{*(n-1)} \overset{\partial_{n-1}}{\to} \ldots \overset{\partial_{2}}{\to} \Z \Lambda^{*1} \overset{\partial_{1}}{\to} \Z \Lambda^{*0} \overset{\partial_{0}}{\to} \Z \overset{0}{\to} \ldots 
	\]
	of categorical chains, as defined in Remark~\ref{rmk:categorical-cohlogy-reformulations} but with $\partial_0(v) = 1$ for any $v \in \Lambda^{*0} = \operatorname{Obj}\Lambda$. This complex is contractible via the homotopy
	\[
		h_n \colon \Z \Lambda^{*n} \to \Z \Lambda^{*(n+1)} \, , \quad ( \lambda_0, \ldots, \lambda_{n-1}) \mapsto (-1)^{n+1} ( \lambda_0, \ldots, \lambda_{n-1}, \alpha_{s(\lambda_{n-1})}) 
	\]
	for $n \geq 0$, and 
	\[
		h_{-1}\colon \Z  \to \Z \Lambda^{*0} \, , \quad 1 \mapsto \alpha \; .
	\]
	In particular, this fact already shows that for $k$-graphs with initial vertices, the cubical homology groups and the categorical homology groups coincide. 
\end{rmk}

\begin{rmk}\label{rmk:alg-of-Lambda-v}
	If a row-finite, source-free $k$-graph $\Lambda$ has an initial vertex $\alpha$, then $C^*({\Lambda}) \cong \K(\ell^2(\operatorname{Obj}\Lambda))$. 
	By definition, $C^*({\Lambda})$ is generated by projections $\{ p_v \colon v \in \operatorname{Obj}\Lambda\}$ and partial isometries $\{s_{\lambda }\colon \lambda \in \Lambda\}$ subject to the relations
	\[
		s_{\lambda}^* s_{\lambda} = p_{s(\lambda) } \, , \quad  \sum_{\lambda \in v \Lambda^n} \quad s_{\lambda} s_{\lambda}^* = p_{v} \, , \quad \text{and}\quad s_{\lambda} s_{\mu} = \delta_{s(\lambda), r(\mu) } s_{\lambda \mu} \; .
	\]
	Note that $p_v = s_{v}$ for any $v \in \operatorname{Obj}\Lambda$.   It follows immediately from these relations that $C^*(\Lambda)$ is densely spanned by $\{s_\lambda s_\mu^*: s(\lambda) = s(\mu)\}$.
	
When $\Lambda$ has an initial vertex $\alpha$, the uniqueness of the elements $\alpha_{v}$ implies that 
$p_v = s_{\alpha_v} s_{\alpha_v}^*$ for any vertex $v$.
Consequently,
\[ s_\lambda s_\mu^* = s_\lambda s_{\alpha_{s(\lambda)}} s_{\alpha_{s(\lambda)}}^* s_\mu^* = s_{\alpha_{r(\lambda)}} s_{\alpha_{r(\mu)}}^*.\]
This tells us that the map from $C^*(\Lambda)$ to $\K(\ell^2(\operatorname{Obj}\Lambda))$ given on the dense spanning set by 
\[ s_\lambda s_\mu^* \mapsto E_{r(\lambda), r(\mu)}\]
is well defined.  It is easily checked to be a $*$-isomorphism.
\end{rmk}

\section{Main results and consequences}\label{sec:main}

\begin{thm}	\label{thm:proj-res}
	The chain complex 
	\begin{equation}\label{eq:cubical-free-res}
		\dots \stackrel{\partial_{n+1}}{\to} \Z Q_{n+1}(\widetilde{\Lambda}) \stackrel{\partial_n}{\to} \Z Q_n(\widetilde{\Lambda}) \stackrel{\partial_{n-1}}{\to} \dots \to \Z Q_0(\widetilde{\Lambda}) \stackrel{\partial_0}{\to} \Z^\Lambda \to 0
	\end{equation}
	of right $\Lambda$-modules is a free resolution of the constant $\Lambda$-module $\Z^\Lambda$. 
\end{thm}

We point out that this chain complex is merely a truncation of the complex $\Z \widetilde{Q}_*(\widetilde{\Lambda})$ introduced in Construction~\ref{constr:cubical-resolution}. We will refer to it as the \emph{cubical free resolution} of $\Lambda$. 

\begin{proof}
	Proposition~\ref{prop:cubical-module-free} tells us that $\Z Q_n(\widetilde{\Lambda})$ is a free $\Lambda$-module for each $n \in \Z$. On the other hand, since each future path $k$-graph $\widetilde{\Lambda}(v)$ has an initial vertex by Remark~\ref{rmk:future-path-graph}\eqref{rmk:future-path-graph:initial}, Corollary~\ref{cor:top-realiz-acyclic} implies that for each $v \in \operatorname{Obj}\Lambda$, the chain complex of abelian groups $\chaincomplex{ \Z \widetilde{Q}_n(\widetilde{\Lambda}(v)), \partial_n }_{n \in \N}$ is exact (see Definition~\ref{def:cubical-chain-complex}). Hence the chain complex $\chaincomplex{\Z \widetilde{Q}_n(\widetilde{\Lambda}), \partial_n }_{n \in \N}$ of $\Lambda$-modules is exact, which is what we wanted to prove. 
\end{proof}

\begin{thm}
\label{thm:cubical-equals-categorical}
For any $k$-graph $\Lambda$, the cubical and categorical \mbox{(co-)}homology groups of $\Lambda$ are isomorphic:
\[
	H_n^{\operatorname{cub}}(\Lambda, \mathcal N) \cong H_n(\Lambda, \mathcal N) \text{ and }   H^n_{\operatorname{cub}}(\Lambda, \mathcal M) \cong H^n(\Lambda, \mathcal M) 
\]
for any left $\Lambda$-module $\mathcal{N}$ and right $\Lambda$-module $\mathcal{M}$. 
\end{thm}
\begin{proof}
	It follows from Theorem~\ref{thm:proj-res} and Definition \ref{def:categorical-cohlogy} that we can use the complex $\Z Q_*(\widetilde{\Lambda})$ to compute the categorical \mbox{(co-)}homology of $\Lambda$:
	\begin{equation} \label{eq:new-computation-categorical}
	H_n(\Lambda, \mathcal{N}) \cong H^n\left( Z Q_*(\widetilde{\Lambda})\otimes_\Lambda \mathcal{N}\right) \text{ and }  H^n(\Lambda, \mathcal{M}) \cong H^n\left( \Hom_\Lambda\left(\Z Q_*(\widetilde{\Lambda}), \mathcal{M}\right) \right) 
	\end{equation}
	for an arbitrary left $\Lambda$-module $\mathcal{N}$ and right $\Lambda$-module $\mathcal{M}$. These are then isomorphic to the cubical \mbox{(co-)}homology groups by Proposition~\ref{prop:cubical-chain-iso}.
\end{proof}

Next we detail a few of the consequences of the isomorphism $H^n_{\operatorname{cub}}(\Lambda, \mathcal M) \cong H^n(\Lambda, \mathcal M) $ established in Theorem \ref{thm:cubical-equals-categorical}
above. 

\begin{prop}
	Let $\Lambda$ be a $k$-graph, $\mathcal{N}$ a left $\Lambda$-module, and $\mathcal{M}$ a right $\Lambda$-module.  The categorical homology groups $H_n(\Lambda, \mathcal N)$ and the categorical cohomology groups $H^n(\Lambda, \mathcal M)$ of $\Lambda$ are zero for $n > k $.  
\end{prop}
\begin{proof}
	By construction, there are no $n$-cubes in a $k$-graph if $n > k$, so $0=H_n^{\operatorname{cub}}(\Lambda, \mathcal N) \cong H_n(\Lambda, \mathcal N) $ and the same goes for the cohomology groups.
\end{proof}

The next proposition uses the notation $M^\Lambda$ from Definition~\ref{def:lambda-module}. 

\begin{prop}
	Let $M$ be an abelian group. Then for all $n$, the categorical homology groups $H_n(\Lambda,  M^\Lambda)$ (respectively, the categorical cohomology groups $H^n(\Lambda,  M^\Lambda)$) are isomorphic to the cohomology groups $H_n(|\Lambda|, M)$ (respectively, the homology groups $H^n(|\Lambda|, M)$) of the topological realization of $\Lambda$.
\end{prop}
\begin{proof}
	The case of the homology groups follows from Theorem \ref{thm:cubical-equals-categorical} and the isomorphism $H_{n}^{\operatorname{cub}}(\Lambda, M) \cong H_n(|\Lambda|, M)$ from \cite[Theorem 6.3]{kps-hlogy}. For cohomology, we use the universal coefficient theorem (UCT) for the cubical (co)homology groups of $k$-graphs established in \cite[Theorem 7.3]{kps-hlogy}: the sequence 
	\[ 0 \to \text{Ext}(H_{n-1}^{\operatorname{cub}}(\Lambda), M) \to H^n_{\operatorname{cub}}(\Lambda, M) \to \Hom(H_{n}^{\operatorname{cub}}(\Lambda), M)\to 0\]
	is exact. Combined with the isomorphisms 
	\[ H_{n}^{\operatorname{cub}}(\Lambda) \cong H_n(|\Lambda|) \quad \text{ and } \quad H^n_{\operatorname{cub}}(\Lambda, M) \cong H^n(\Lambda, M^\Lambda),\]
	 the UCT exact sequence becomes
	\[ 0 \to \text{Ext}(H_{n-1}(|\Lambda|), M) \to H^n(\Lambda, M^\Lambda) \to \Hom(H_n(|\Lambda|), M) \to 0.\]
	The UCT for topological spaces and the Five Lemma now imply that 
	\[ H^n(\Lambda, M^\Lambda) \cong H^n(|\Lambda|, M). \qedhere \]
\end{proof}

\section{The chain maps}
\label{sec:chain-maps}

In this last section, we provide a second, algebraic, proof of the exactness of the cubical free resolution in Equation~\eqref{eq:cubical-free-res}, without using the topological constructions and results from Section~\ref{sec:initial-vertices}. The advantage of our second approach is that it constructs explicit \mbox{(co-)}chain maps  that implement the isomorphisms between the categorical and cubical \mbox{(co-)}homology groups, as these groups were originally defined in \cite{kps-hlogy}. To be precise, this second proof does not rely on the result  \cite[Theorem III.6.3]{maclane} that allows one to use any projective resolution to compute the categorical \mbox{(co-)}homology. We anticipate that this approach may facilitate future computations. As examples of such computations, we establish naturality of our isomorphisms and  compare our isomorphisms with those constructed in \cite{kps-twisted} in degrees 0, 1, and 2. 

These \mbox{(co-)}chain maps ultimately come from $\Lambda$-chain maps back and forth between the cubical free resolution 
\[
	\dots \stackrel{\partial_{n+1}}{\to} \Z Q_{n+1}(\widetilde{\Lambda}) \stackrel{\partial_n}{\to} \Z Q_n(\widetilde{\Lambda}) \stackrel{\partial_{n-1}}{\to} \dots \to \Z Q_0(\widetilde{\Lambda}) \stackrel{\partial_0}{\to} \Z^\Lambda \overset{0}{\to} 0
\]
from Equation~\eqref{eq:cubical-free-res} and the standard categorical free resolution 
\[
	\ldots \overset{\partial_{n+1}^{\Presolution}}{\to} \mathcal P_{n}(\Lambda) \overset{\partial_{n}^{\Presolution}}{\to} \mathcal P_{n-1}(\Lambda) \overset{\partial_{n-1}^{\Presolution}}{\to} \ldots  \overset{\partial_{1}^{\Presolution}}{\to} \mathcal  P_{0}(\Lambda) \overset{\partial_{0}^{\Presolution}}{\to} \Z^\Lambda \overset{0}{\to} 0
\]
from Remark~\ref{rmk:categorical-resolution}. 

Viewing these as bi-infinite $\Lambda$-chain complexes, we keep the notations $\Z \widetilde{Q}_*(\widetilde{\Lambda})$ for the former and $\Presolution_*(\Lambda)$ for the latter. We remind the reader that in both complexes, the right action of $\Lambda$ affects only the last component of a generating tuple; in other words, each generator's pertinent homological information is carried in the previous entries in the tuple.  Consequently, our $\Lambda$-chain maps  will always leave the last entry in each tuple untouched.


\begin{defn}
	Let $\Lambda$ be a $k$-graph and let $\lambda \in \Lambda$. Recall that for $1 \leq i \leq k$, $d_i(\lambda) \in \N $ is defined so that $d(\lambda) = \sum_{i=1}^{k} d_i(\lambda) e_i \in \N^k$. We also define $C(\lambda)$, the \emph{color type} of $\lambda$, to be the set $\{ i \in \{1, \ldots, k\} \colon d_i(\lambda) > 0 \}$. Lastly, we let $C_{(j)}(\lambda)$ be the $j$th smallest number in $C(\lambda)$, for $j = 1, \ldots, |C(\lambda)|$. We sometimes simply write $d_i$, $C$ and $C_{(j)}$ if it is clear which morphism we are working with. 
\end{defn}

\subsection{Mapping cubes to composable tuples}

We first fix some notation. The symmetric group on $\{1, \ldots, n\}$ is denoted by $\Sigma_{n}$, so that $\Sigma_{n-1}$ is a subgroup of $\Sigma_{n}$. For a permutation $\sigma \in \Sigma_{n}$, we write $\operatorname{sgn}(\sigma) $ for its sign, which takes value in $\{-1, 1\}$. 
  
\begin{constr}\label{constr:triangle-map}
	Given $v \in \operatorname{Obj}\Lambda$, $(\lambda, \mu) \in Q_n(\widetilde{\Lambda})(v)$ and a permutation $\sigma \in \Sigma_{n}$, write 
	\[ \lambda = \lambda_1^\sigma \cdots\lambda_n^\sigma\]
	where, for each $i$, we have $C(\lambda_i^{\sigma}) = \{ C_{(\sigma(i))} (\lambda) \}$, that is, $\lambda_i^{\sigma}$ only carries the $\sigma(i)$-th color of $\lambda$. The existence of such a decomposition is guaranteed by the factorization property. In fact, using the notation in Remark~\ref{rmk:Lambda-m-n}, we can write 
	\[
		\lambda_i^{\sigma} = \lambda \left( \sum_{j=1}^{i-1} d_{C_{(\sigma(j))}}(\lambda) e_{C_{(\sigma(j))}}, \sum_{j=1}^{i} d_{C_{(\sigma(j))}}(\lambda)   e_{C_{(\sigma(j))}}\right) \; .
	\] 
	
	For any $n \in \N$, we define a $\Lambda$-module map $\triangledown_n \colon \Z Q_n(\widetilde{\Lambda}) \to \mathcal P_n(\Lambda)$ by linearly extending the prescription
	\begin{equation}
		\label{eq:triangle-formula} \left(\triangledown_n\right)_{v} (\lambda, \mu) := \sum_{\sigma \in \Sigma_{n}} \operatorname{sgn}(\sigma) (\lambda_1^\sigma, \ldots, \lambda_n^\sigma, \mu) 
	\end{equation}
	for any $v \in \operatorname{Obj}\Lambda$ and $(\lambda, \mu) \in Q_n(\widetilde{\Lambda}(v))$. 
	We define $\triangledown_{-1} \colon \Z \widetilde{Q}_{-1}(\widetilde{\Lambda}) \to \Presolution_{-1}(\Lambda)$ to be the identity map on $\Z^\Lambda$. For $n < -1$, we let $\triangledown_n \colon \Z \widetilde{Q}_{n}(\widetilde{\Lambda}) \to \Presolution_{n}(\Lambda)$ be the zero map. 
\end{constr} 

\begin{prop}\label{prop:triangle-is-chain-map}
	The sequence of maps $\left(\triangledown_n\colon \Z {\widetilde{Q}}_n( \widetilde{\Lambda}) \to \Presolution_n(\Lambda)\right)_{n \in \Z}$ constitutes a map of chain complexes, that is, $\triangledown_{n-1} \circ \partial_{n} = \partial_n^{\Presolution} \circ \triangledown_{n}$ for all $n \in \Z$.
\end{prop}
\begin{proof}
	The statement is obvious for $n < 0$. The case for $n = 0$ follows from the observation that $\left(\triangledown_0\right)_v (\mu)= \mu$ for any $v \in \operatorname{Obj}\Lambda$ and any vertex $ \mu\in \operatorname{Obj} \widetilde{\Lambda}(v)$. Thus we focus on the case for an arbitrary $n > 0$. 
	
	For any $j \in \{1, \ldots, n-1\}$, let $\rho_j \in \Sigma_{n}$ be the transposition $(j, j+1)$, which satisfies $\operatorname{sgn}(\rho_j) = -1$ and $\Sigma_{n} = A_n \sqcup A_n \rho_j$, where $A_n$ denotes the alternating subgroup in $\Sigma_{n}$. We notice that for any permutation $\sigma \in \Sigma_{n}$ and any $\lambda \in Q_n(\Lambda)$, we have
	\[ 
		\lambda_i^\sigma = \lambda_i^{\sigma \rho_j} \text{ for all } \ i \not\in \{j, j+1\}; \qquad \quad \lambda_j^\sigma \lambda_{j+1}^\sigma = \lambda_j^{\sigma \rho_j} \lambda_{j+1}^{\sigma \rho_j}.
	\]
	Thus for any $v \in \operatorname{Obj}\Lambda$ and $(\lambda, \mu) \in Q_n(\widetilde{\Lambda}(v))$, we have
	\begin{align*}
		\left(\partial_n^{\Presolution} \circ \triangledown_n\right)_v (\lambda, \mu) &= \sum_{\sigma \in \Sigma_{n}} \operatorname{sgn}(\sigma) \left( \vphantom{\sum_{1}^{2}} (\lambda_2^\sigma, \ldots, \lambda_n^\sigma, \mu) + (-1)^n (\lambda_1^\sigma, \ldots, \lambda_n^\sigma \mu)  \right. \\
		& \qquad\qquad\qquad  \left. + \sum_{j=1}^{n-1} (-1)^j (\lambda^\sigma_1, \ldots, \lambda_j^{\sigma} \lambda_{j+1}^\sigma, \ldots, \lambda_n^\sigma, \mu) \right) \\
		&= \sum_{\sigma \in \Sigma_{n}} \operatorname{sgn}(\sigma) \Big( \vphantom{1^2} (\lambda_2^\sigma, \ldots, \lambda_n^\sigma, \mu) + (-1)^n (\lambda_1^\sigma, \ldots, \lambda_n^\sigma \mu)  \Big) \\
		& \quad  + \sum_{j=1}^{n-1} (-1)^j \left( \sum_{\sigma \in A_{n}} (+1) \left(\lambda^\sigma_1, \ldots, \lambda_j^{\sigma} \lambda_{j+1}^\sigma, \ldots, \lambda_n^\sigma, \mu \right) \right. \\
		& \qquad \qquad \qquad \left. \vphantom{\sum_{\sigma \in A_{n}}} + (-1) \left(\lambda^{\sigma\rho_j}_1, \ldots, \lambda_j^{\sigma\rho_j} \lambda_{j+1}^{\sigma\rho_j}, \ldots, \lambda_n^{\sigma\rho_j}, \mu \right)  \right) \\
		&= \sum_{\sigma \in \Sigma_{n}} \operatorname{sgn}(\sigma) \Big( (\lambda_2^\sigma, \ldots, \lambda_n^\sigma, \mu) + (-1)^n (\lambda_1^\sigma, \ldots, \lambda_n^\sigma \mu) \Big).
	\end{align*}
	
	On the other hand, for any $j \in \{1, \ldots, n-1\}$, let $\kappa_j \in \Sigma_{n}$ be the cyclic permutation $(j, j+1, \ldots, n)$ and $\kappa_n = \operatorname{id} \in \Sigma_{n}$. Thus $\operatorname{sgn}(\kappa_j) = (-1)^{n-j}$ and $\Sigma_{n} = \bigsqcup_{j=1}^{n}  \kappa_j \Sigma_{n-1} = \bigsqcup_{j=1}^{n}  \kappa_j \Sigma_{n-1}  \kappa_1^{-1}$. More concretely, we have
	\[
		\kappa_j (i) =\begin{cases} i, & i < j \\ i+1, & j \leq i \leq n-1 \\  j, & i = n \end{cases} \; .
	\]	
	Notice that for any $(\lambda, \mu) \in Q_n(\widetilde{\Lambda}(v))$, we have
	\[ 
		\left(\partial_n\right)_v (\lambda, \mu) = \sum_{j=1}^n (-1)^j \Big( \left(F^0_j(\lambda), S_j(\lambda) \mu \right) - \left(F^1_j(\lambda), \mu \right) \Big) \; ,
	\]
	where, in our current notation, we have
	\[
		F^0_j(\lambda) = \lambda^{\kappa_j}_1 \ldots \lambda^{\kappa_j}_{n-1} , \ S_j(\lambda) = \lambda^{\kappa_j}_{n} , \ R_j(\lambda) = \lambda^{\kappa_j \kappa_1^{-1}}_{1} \text{and } F^1_j(\lambda) = \lambda^{\kappa_j \kappa_1^{-1}}_{2} \ldots \lambda^{\kappa_j \kappa_1^{-1}}_{n \vphantom{2}} \; .
	\]
	This implies, for any $\tau \in \Sigma_{n-1}$ and any $i \in \{1, \ldots, n-1\}$, that $\left(F^0_j(\lambda)\right)^{\tau}_i = \lambda^{\kappa_j{\tau}}_i$ and $\left(F^1_j(\lambda)\right)^{\tau}_{i} = \lambda^{\kappa_j{\tau} \kappa_1^{-1}}_{i+1}$. 
	Therefore, we compute
	\begin{align*}
		& \hspace{-.7cm}  \left(\triangledown_{n-1} \circ \partial_n\right)_v (\lambda, \mu) \\
		&= \sum_{j=1}^n (-1)^{j} \sum_{\tau \in \Sigma_{n-1}} \operatorname{sgn}(\tau) \bigg( \left( \left(F^0_j(\lambda)\right)^{\tau}_1 , \ldots, \left(F^0_j(\lambda)\right)^{\tau}_{n-1}, S_j(\lambda) \mu \right) \\
		& \qquad\qquad\qquad\qquad\qquad\qquad - \left( \left(F^1_j(\lambda)\right)^{\tau}_1 , \ldots, \left(F^1_j(\lambda)\right)^{\tau}_{n-1}, \mu \right) \bigg)  \\	
		&= \sum_{j=1}^n (-1)^{j} \sum_{\tau \in \Sigma_{n-1}} \operatorname{sgn}(\tau) \bigg( \left( \lambda^{\kappa_j{\tau}}_1 , \ldots, \lambda^{\kappa_j{\tau}}_{n-1}, \lambda^{\kappa_j}_{n} \mu \right) \\
		& \qquad\qquad\qquad\qquad\qquad\qquad - \left( \lambda^{\kappa_j{\tau} \kappa_1^{-1}}_2 , \ldots, \lambda^{\kappa_j{\tau} \kappa_1^{-1}}_{n}, \mu \right) \bigg)  \\	
		&= \sum_{j=1}^n \sum_{\tau \in \Sigma_{n-1}} \bigg( (-1)^{n} \operatorname{sgn}(\kappa_j\tau) \left( \lambda^{\kappa_j{\tau}}_1 , \ldots, \lambda^{\kappa_j{\tau}}_{n-1}, \lambda^{\kappa_j{\tau}}_{n} \mu \right) \\
		& \qquad\qquad\qquad + \operatorname{sgn}(\kappa_j\tau\kappa_1^{-1}) \left( \lambda^{\kappa_j{\tau} \kappa_1^{-1}}_2 , \ldots, \lambda^{\kappa_j{\tau} \kappa_1^{-1}}_{n}, \mu \right) \bigg)  \\
		&= \sum_{\sigma \in \Sigma_{n}} (-1)^n \operatorname{sgn}(\sigma)(\lambda_1^\sigma, \ldots, \lambda_n^\sigma \mu) + \sum_{\sigma' \in \Sigma_{n}}  \operatorname{sgn}(\sigma')  (\lambda_2^{\sigma'}, \ldots, \lambda_n^{\sigma'}, \mu)   \\
		&= \left(\partial_n^{\Presolution} \circ \triangledown_n\right)_v (\lambda, \mu) \; .
	\end{align*}
	This proves $\triangledown_{n-1} \circ \partial_n = \partial_n^{\Presolution} \circ \triangledown_n$ for all $n \in \Z$.
\end{proof}

\subsection{Mapping composable tuples to cubes}
In this subsection, we construct a chain map $\boxempty_*\colon \Presolution_*(\Lambda) \to \Z \widetilde{Q}_*(\widetilde{\Lambda})$. To streamline the notation, we carry out computations of cubical chains in terms of general rectangles instead of just cubes. Let us specify how we do this. 

\newcommand{\rectangularchain}[1]{\left\lceil {#1} \right\rfloor}

\begin{defn}
	Recall that for a morphism $\lambda$ in a $k$-graph, $C(\lambda)$ is the color type of $\lambda$. For any $K \subset \{1, \ldots, k\}$, we write $\Lambda^{K} = \{ \lambda \in \Lambda \colon C(\lambda) \subset K \}$. For any $n \in \N$, we let ${k \choose n}$ denote the collection of all subsets of $\{1, \ldots, k\}$ with cardinality $n$. Note that this is empty if $n > k$. For any $K \in {k \choose n}$ and any $i \in \{1, \ldots, n\}$, we write $K_{(i)}$ for the $i$th smallest number in $K$. 
	
	Also recall from Remark~\ref{rmk:Lambda-m-n} that for any $n,m \in \mathbb{N}^k$ with $0 \leq n \leq m \leq d(\lambda)$, we write $\lambda(n,m)$ for the image of $(n,m) \in \Omega^k_{\leq d(\lambda)}$ under the canonical $k$-graph morphism $\Omega^k_{\leq d(\lambda)} \to \Lambda$. For example, we have $\lambda (0, 0) = r(\lambda)$, $\lambda (0, d(\lambda)) = \lambda$, and $\lambda (d(\lambda), d(\lambda)) = s(\lambda)$. 
\end{defn}

\begin{defn}\label{def:rectangular-chain}
	For any $n \in \N$, any $K \in {k \choose n}$ and any $\lambda \in \Lambda^K$, writing $e_K = \sum_{j \in K}{e_j} \in \N^k$, we define $\rectangularchain{\lambda}_K \in \Z Q_n(\Lambda)$, the \emph{$K$-shaped rectangular chain} associated to $\lambda$, to be the sum
	\[
	\rectangularchain{\lambda}_K:= 	\sum_{m \in [0, d(\lambda)-e_K] \cap \N^k} \lambda(m, m+e_K)
	\]
	That is, we compute	$\rectangularchain{\lambda}_K$ by dividing the hyper-rectangle $[0, d(\lambda)]$ into $n$-dimensional unit cubes and evaluating $\lambda$ on each one.

	Observe that if $|C(\lambda)| < n$, then $ \rectangularchain{\lambda}_K = 0$ for any $K \in {k \choose n}$, in which case we say $ \rectangularchain{\lambda}_K$ is a \emph{degenerate} $K$-shaped rectangular chain.
\end{defn}

\begin{defn}\label{def:rectangular-faces}
	For any $n \in \N$ and any $K \in {k \choose n}$,  
	we define, for each $j \in \{1, \ldots, n\}$, 
	\begin{align*}
	\widehat{F}^0_{K,j} &\colon  \Lambda^K \to \Lambda^{K \setminus \{C_{(j)}\}} \, , & \lambda &\mapsto \lambda\big(0, d(\lambda) - d_{C_{(j)}}(\lambda) e_{C_{(j)}}\big) \, , \\
	S_{K,j} &\colon  \Lambda^K \to \Lambda^{ \{C_{(j)}\}} \, , & \lambda &\mapsto  \lambda\big(d(\lambda) - d_{C_{(j)}}(\lambda)e_{C_{(j)}}, d(\lambda)\big) \, , \\
	\widehat{F}^1_{K,j} &\colon  \Lambda^K \to \Lambda^{K \setminus \{C_{(j)}\}} \, , & \lambda &\mapsto  \lambda\big(d_{C_{(j)}} (\lambda) e_{C_{(j)}}, d(\lambda)\big)  \, , \\
	R_{K,j} &\colon  \Lambda^K \to \Lambda^{ \{C_{(j)}\}} \, , & \lambda &\mapsto  \lambda\big(0,  d_{C_{(j)}}(\lambda) e_{C_{(j)}}\big) \, .
	\end{align*}
	Thus $\lambda = \widehat{F}^0_{K,j}(\lambda) \, S_{K,j}(\lambda) =  R_{K,j} (\lambda) \, \widehat{F}^1_{K,j}(\lambda)$. We also write
	\[
		F^l_{K,j}(\lambda)=	\rectangularchain{\widehat{F}^l_{K,j}(\lambda)}_{K\backslash \{ C_{(j)}\}} \text{ for } l = 0,1 \; .
	\]
\end{defn}

\begin{lemma}\label{lem:rectangular-bdry}
	For any $n \in \N$, any $K \in {k \choose n}$, and any $\lambda \in \Lambda^K$, we have
	\[
		\partial_n \left( \rectangularchain{\lambda}_K \right) = \sum_{j=1}^n \sum_{l=0}^1(-1)^{j+l} {{F}^l_{K,j}(\lambda)}
	\]
\end{lemma}

\begin{proof}
We compute:
\begin{align*}
 \partial_n ( \rectangularchain{\lambda}_K) &= \sum_{m \in [0, d(\lambda) - e_K] \cap \N^k} \partial_n(\lambda(m, m+e_K)) \\
 &= \sum_m \sum_{j=1}^n (-1)^j \left( \lambda(m, m+ e_{K \backslash C_{(j)}}) - \lambda(m+ e_{C_{(j)}}, m + e_K) \right)\\
 &= \sum_{j=1}^n \sum_{l =0}^1 (-1)^{j+l} F^l_{K, j}(\lambda). \qedhere
 \end{align*}

\end{proof}

\begin{lemma}\label{lem:rectangular-bdry-concat}
	For any $n \in \N$, any $K \in {k \choose n}$, any $j \in \{1, \ldots, n\}$, and any $\lambda, \mu \in \Lambda^K$, if $\widehat{F}^1_{K,j}(\lambda) = \widehat{F}^0_{K,j}(\mu)$, then we have $\lambda \, S_{K,j}(\mu) = R_{K,j}(\lambda) \, \mu$ and 
	\begin{equation}
		\rectangularchain{\lambda}_K + \rectangularchain{\mu}_K = \rectangularchain{\lambda \, S_{K,j}(\mu)}_K = \rectangularchain{R_{K,j}(\lambda) \, \mu}_K \; . 
	\label{eq:summing-rectangles}
	\end{equation}
\end{lemma}

Intuitively speaking, this lemma says $\lambda$ and $\mu$ can be glued along their common face, which is an $(n-1)$-dimensional hyperrectangle given by $\widehat{F}^1_{K,j}(\lambda) = \widehat{F}^0_{K,j}(\mu)$, to form a larger $n$-dimensional hyperrectangle.

\begin{proof}

The first assertion is an immediate consequence of Definition \ref{def:rectangular-faces}, while the second comes from a direct computation using Definition~\ref{def:rectangular-chain} and~\ref{def:rectangular-faces}. 
\end{proof}

We are now in a position to define the maps $\boxempty_n$, for $n \in \Z$.  These maps are described pictorially in the diagram of Remark \ref{rmk:box-diagram}, which may help the reader to follow the construction below.

\begin{constr}\label{constr:box-map}
	For any $n \in \N$, any $v \in \operatorname{Obj}\Lambda$, any composable $(n + 1)$-tuple $(\lambda_0, \lambda_1, \ldots, \lambda_n)$ in $\Lambda^{* (n+1)}v$, and any $K \in {k \choose n}$, we define 
	\[
		\widehat{\boxempty}_n (\lambda_0, \ldots, \lambda_n; K) := \left( \lambda(b, c),  \lambda(c, d(\lambda)) \right) \in \widetilde{\Lambda} (v)  \; ,
	\]
	where $\lambda = \lambda_0 \cdots \lambda_n$, 
	\[
		b =  \sum_{i=1}^{n} \sum_{j=K_{(i)}+1 }^{k} d_{j}(\lambda_{i-1}) e_j \quad\text{ and }\quad c = \sum_{i=1}^{n} \sum_{j= K_{(i)}}^{k}  d_{j}(\lambda_{i-1}) e_j .
	\]
	Recall that $K_{(i)}$ denotes the $i$th smallest entry in $K$.
	
	We sometimes write $\widehat{\boxempty}_{n,v}$ to emphasize $v$. Note that 
	\begin{equation}\label{eq:box-map-degree}
		\widetilde{d}(\widehat{\boxempty}_n(\lambda_0, \ldots, \lambda_n; K))  = c-b = \sum_{i=1}^{n}  d_{K_{(i)}}(\lambda_{i-1}) e_{K_{(i)}}\; ,
	\end{equation}
	and thus $C (\widehat{\boxempty}_n (\lambda_0, \lambda_1, \ldots, \lambda_n; K) ) \subset K$, which allows us to define
	\[
		\boxempty_n (\lambda_0,  \ldots, \lambda_n; K) := \rectangularchain{\widehat{\boxempty}_n (\lambda_0,  \ldots, \lambda_n; K)}_K \in \Z \widetilde{Q}_n(\widetilde{\Lambda}(v)) \; .
	\] 
	Hence we can define a homomorphism 
	\[
		\left( \boxempty_n \right)_v \colon \Presolution_*(\Lambda)(v) = \Z \Lambda^{*(n+1)} \to \Z \widetilde{Q}_*(\widetilde{\Lambda} (v)) \; .
	\]
	by linearly extending the prescription
	\begin{equation}\label{eq:box-map-sum-K}
		\left( \boxempty_n \right)_v (\lambda_0,  \ldots, \lambda_n) = \sum_{ { K \in {k \choose n} } } \boxempty_n (\lambda_0,  \ldots, \lambda_n; K) \; .		
	\end{equation}
	It is clear that for any $\mu \in v \Lambda$, we have
	\begin{equation}\label{eq:box-map-module}
		\widehat{\boxempty}_n (\lambda_0,  \ldots, \lambda_n \mu; K) = \widehat{\boxempty}_n (\lambda_0,  \ldots, \lambda_n; K) \cdot \widetilde{\Lambda}(\mu) \;.		
	\end{equation}
	Thus for each $n \in \N$, we have constructed a $\Lambda$-module map
	\[
		\boxempty_n \colon \Presolution_n(\Lambda)  \to \Z \widetilde{Q}_n(\widetilde{\Lambda}) \; .
	\]
	For $n = -1$, we define $\boxempty_n \colon \Presolution_n(\Lambda)  \to \Z \widetilde{Q}_n(\widetilde{\Lambda})$ to be the identity map on $\Z^{\Lambda}$. We also define $\boxempty_n$ to be the zero map for $n <-1$.
\end{constr}

We observe that $\boxempty_n = 0$ if $n > k$. 

\begin{rmk}\label{rmk:box-diagram}

	The following diagram may help with visualizing $\widehat{\boxempty}_n(\lambda_0, \ldots, \lambda_n; K)$.    
	\begin{equation}
		\begin{array}{ccccc}
		\multicolumn{1}{|c}{d_{1}(\lambda_{0})} & d_{1}(\lambda_{1}) & \cdots & {d_{1}(\lambda_{n-1})}  & d_{1}(\lambda_{n}) \\
		\multicolumn{1}{|c}{\vdots} & \vdots & \cdots  & \vdots  & \vdots \\
		\cline{1-1} 
		\multicolumn{1}{|c|}{d_{K_{(1)}}(\lambda_{0})} & \vdots & \cdots  & \vdots  & \vdots \\
		\cline{1-1} 
		\multicolumn{1}{c|}{\vdots} & \vdots & \cdots  & \vdots  & \vdots \\
		\cline{2-2} 
		\vdots & \multicolumn{1}{|c|}{d_{K_{(2)}}(\lambda_{1})} & \cdots  & \vdots  & \vdots \\
		\cline{2-2} 
		\vdots & \multicolumn{1}{c|}{\vdots} & \cdots  & \vdots  & \vdots \\
		\vdots & \vdots &  {\cdots} & \vdots  & \vdots \\
		\vdots & \vdots &  \multicolumn{1}{c|}{\cdots} & {\vdots}  & \vdots \\
		\cline{4-4} 
		\vdots & \vdots & \cdots  & \multicolumn{1}{|c|}{d_{K_{(n)}}(\lambda_{n-1})}  &   \vdots \\
		\cline{4-4} 
		\vdots & \vdots & \cdots  & \multicolumn{1}{c|}{\vdots}   & \vdots \\
		d_{k}(\lambda_{0}) & d_{k}(\lambda_{1}) & \cdots & d_{k}(\lambda_{n-1})   & \multicolumn{1}{|c}{  d_{k}(\lambda_{n})} 
		\end{array}
		\label{eq:diagram-box}
	\end{equation}
	\begin{itemize}
		\item The degree of $\lambda(b, c)$ in $\Lambda$, or equivalently, that of $\widehat{\boxempty}_n(\lambda_0, \ldots, \lambda_n; K)$ in $\widetilde{\Lambda}(v),$ is given by adding up the entries in the boxes. 
		\item The degree of $\lambda(c, d(\lambda))$ in $\Lambda$, which acts as the source of $\widehat{\boxempty}_n(\lambda_0, \ldots, \lambda_n; K) $ in $\widetilde{\Lambda}(v)$, is given by adding up the entries above and to the right of the boxes. 
		\item If we add up the two sums above, we obtain the degree of $\lambda(b, d(\lambda))$ in $\Lambda$, which acts as the range of $\widehat{\boxempty}_n(\lambda_0, \ldots, \lambda_n; K) $ in $\widetilde{\Lambda}(v)$.
	\end{itemize}
\end{rmk}

Our main task in this subsection is to prove the following: 

\begin{thm}\label{thm:box-is-chain-map}
	The sequence of maps $\left(\boxempty_n \colon \Presolution_n(\Lambda) \to \Z {\widetilde{Q}}_n( \widetilde{\Lambda}) \right)_{n \in \Z}$ constitutes a map of chain complexes, that is, $\boxempty_{n-1} \circ \partial_{n}^{\Presolution} = \partial_n \circ \boxempty_{n}$ for all $n \in \Z$.
\end{thm}

For the proof, we will need to describe the elements in $\Z {\widetilde{Q}}_{n-1}( \widetilde{\Lambda})$ which arise when we compute  $\boxempty_{n-1} \circ \partial_{n}^{\Presolution} (\lambda_0, \ldots, \lambda_n)$ and $ \partial_n \circ \boxempty_{n}(\lambda_0, \ldots, \lambda_n)$.   These $(n-1)$-cubes, which we denote by $\widehat  \Xi(\lambda_0, \ldots, \lambda_n; J, q)$, are constructed as follows.
\begin{constr}\label{constr:box-face-map}
	For any $J \in {k \choose {n-1}}$, recall that $J_{(i)}$ is the $i$th smallest number in $J$ for $i \in \{1, \ldots, n-1 \}$. We also write $J_{(0)} = 0$ and $J_{(n)} = k+1$. For  any $q\in \{0, 1, \ldots, k \}$, define
	\[
		m(J, q) = \max\{i \in \{0, \ldots, n-1\} \colon  J_{(i)} \leq q \} \; .
	\]
	Observe that for any $J \in {k \choose n-1}$, any $l \in \{0, \ldots, n-1\}$ and any $q \in \{0, \ldots, k\}$, 
	\begin{equation} \label{eq:m-J-q-l}
		m(J, q) = l \text{~if~and~only~if~} J_{(l)} \leq q < J_{(l+1)} \; .
	\end{equation}
	In particular, we have
	\begin{align}
		\label{eq:m-J-0} m(J, 0) &= 0 \; , \\
		\label{eq:m-J-k} m(J, k) &= n-1  \; ,\\
		\label{eq:m-J-q} m(J, q-1) &= 	\begin{cases}
											m(J, q) & \text{if~} q \not\in J \; , \\
											m(J, q)-1  & \text{if~} q \in J \; . 
										\end{cases} 
	\end{align}
	For any $J \in {k\choose{n-1}}$, $q \in \{ 0, \ldots, n-1\}$, any $v \in \operatorname{Obj}\Lambda$ and any $(\lambda_0,  \ldots, \lambda_n) \in \Lambda^{*(n+1)}v$, writing $m$ for $m(J, q)$, we define $\widehat{\Xi}(\lambda_0,  \ldots, \lambda_n ; J, q)$ to be the unique morphism in $\widetilde{\Lambda}(s(\lambda_n))$ with range 
	\[
		\lambda\left( \sum_{i =1}^{m} \sum_{j=J_{(i)}+1}^{k} d_{j}(\lambda_{i-1}) e_j + \sum_{j=q+1}^k d_j(\lambda_{m}) e_j + \sum_{i =m+1}^{n-1} \sum_{j=J_{(i)}+1}^{k} d_{j}(\lambda_{i}) e_j \, , ~ d(\lambda_0 \cdots \lambda_n) \right)
	\]
	and source 
	\[
		\lambda\left( \sum_{i =1}^{m} \sum_{j=J_{(i)}}^{k } d_{j}(\lambda_{i-1}) e_j + \sum_{j=q+1}^k d_j(\lambda_{m}) e_j + \sum_{i =m+1}^{n-1} \sum_{j=J_{(i)}}^{k} d_{j}(\lambda_{i}) e_j\, , ~  d(\lambda_0 \cdots \lambda_n) \right) \; .
	\]
	Note that 
	\[
		\widetilde{d}(\widehat{\Xi}(\lambda_0,  \ldots, \lambda_n ; J, q))  = \sum_{i =1}^{m} d_{J_{(i)}}(\lambda_{i-1}) e_{J_{(i)}} + \sum_{i =m+1}^{n-1} d_{J_{(i)}}(\lambda_{i})e_{J_{(i)}}
	\]
	and thus $C (\widehat{\Xi}(\lambda_0,  \ldots, \lambda_n ; J, q) ) \subset J$, which allows us to define
	\[
		\Xi(\lambda_0,  \ldots, \lambda_n ; J, q) := \rectangularchain{\widehat{\Xi}(\lambda_0,  \ldots, \lambda_n ; J, q)}_{J} \in \Z \widetilde{Q}_{n-1}(\widetilde{\Lambda}(v)) \; .
	\] 
\end{constr}

\begin{rmk}\label{rmk:box-face-diagram}
Diagrammatically, the source $\widetilde{s} \left( \widehat{\Xi}(\lambda_0,  \ldots, \lambda_n ; J, q) \right)$ is given by evaluating $\lambda$ on the sum of the entries above
and to the right of the boxes in the diagram below; the range $\widetilde{r} \left( \widehat{\Xi}(\lambda_0,  \ldots, \lambda_n ; J, q) \right)$ is obtained by further adding in the value of $\lambda$ on the boxes. 
\[
\begin{array}{cccccccc}
\multicolumn{1}{|c}{d_{1}(\lambda_{0})} & \cdots & d_{1}(\lambda_{m-1}) & d_{1}(\lambda_{m}) & d_{1}(\lambda_{m+1}) & \cdots & {d_{1}(\lambda_{n-1})}  & d_{1}(\lambda_{n}) \\
\multicolumn{1}{|c}{\vdots} & \cdots & \vdots & \vdots  & \vdots & \cdots & \vdots  & \vdots \\
\cline{1-1} 
\multicolumn{1}{|c|}{d_{J_{(1)}}(\lambda_{0})} & \cdots & \vdots & \vdots  & \vdots & \cdots & \vdots  & \vdots \\
\cline{1-1} 
\multicolumn{1}{c|}{\vdots} & \cdots & \vdots & \vdots  & \vdots & \cdots & \vdots  & \vdots \\
\vdots & \cdots & \vdots & \vdots  & \vdots & \cdots & \vdots  & \vdots \\
\vdots & \multicolumn{1}{c|}{\cdots} & \vdots & \vdots  & \vdots & \cdots & \vdots  & \vdots \\
\cline{3-3} 
\vdots & \cdots  & \multicolumn{1}{|c|}{d_{J_{(m)}}(\lambda_{m-1})} & \vdots & \vdots & \cdots & \vdots  & \vdots \\
\cline{3-3} 
\vdots & {\cdots} & \multicolumn{1}{c|}{\vdots} & \vdots  & \vdots & \cdots & \vdots  & \vdots \\
\vdots & \cdots  & \multicolumn{1}{c|}{\vdots} & d_{q}(\lambda_{m}) &  \multicolumn{1}{c}{\vdots} & \cdots & \vdots  & \vdots \\
\cline{4-4} 
\vdots & {\cdots} & \vdots & \multicolumn{1}{c|}{\vdots} &  \vdots & \cdots & \vdots  & \vdots \\
\cline{5-5} 
\vdots & \cdots  & \vdots &  \vdots & \multicolumn{1}{|c|}{d_{J_{(m+1)}}(\lambda_{m+1})} & {\cdots} & \vdots  & \vdots \\
\cline{5-5} 
\vdots & {\cdots} & \vdots &  \vdots & \multicolumn{1}{c|}{\vdots} & \cdots & \vdots  & \vdots \\
\vdots  & \cdots  & \vdots & \vdots & \vdots & \cdots  & \vdots  & \vdots \\
\vdots  & \cdots  & \vdots & \vdots & \vdots & \cdots  & \multicolumn{1}{|c}{\vdots}  & \vdots \\
\cline{7-7} 
\vdots  & \cdots  & \vdots & \vdots & \vdots & \cdots  & \multicolumn{1}{|c|}{d_{J_{(n-1)}}(\lambda_{n-1})}  &   \vdots \\
\cline{7-7} 
\vdots  & \cdots  & \vdots & \vdots & \vdots & \cdots  & \multicolumn{1}{c|}{\vdots}   & \vdots \\
d_{k}(\lambda_{0}) & \cdots & d_{k}(\lambda_{m-1}) & d_{k}(\lambda_{m}) & d_{k}(\lambda_{m+1}) & \cdots & d_{k}(\lambda_{n-1})   & \multicolumn{1}{|c}{d_{k}(\lambda_{n})} 
\end{array}
\]
\end{rmk}

In the following, Lemmas \ref{lem:8-15} and \ref{lem:8-17} respectively establish that, as claimed, the rectangles $\Xi(\lambda_0, \ldots, \lambda_n; J, q)$ appear in $\partial_n \circ \boxempty_n(\lambda_0, \ldots, \lambda_n)$ and $\boxempty_{n-1} \circ \partial_n^{\Presolution}(\lambda_0, \ldots, \lambda_n)$. 
The reader is encouraged to use the above diagrams to follow the proofs of these lemmas. 

\begin{lemma}	
\label{lem:8-15}
	For any $n \in \N$, any $v \in \operatorname{Obj}\Lambda$, any $(\lambda_0, \lambda_1, \ldots, \lambda_n) \in \Lambda^{* (n+1)}v$, any $K \in {k \choose n}$, any $l \in \{ 1, \ldots, n \}$ and any $\varepsilon \in \{0,1\}$, we have
	\begin{equation} \label{eq:Fbox-K}
		F^\varepsilon_{K, l}( \widehat{\boxempty}_n(\lambda_0, \ldots, \lambda_n; K)) = \Xi(\lambda_0,  \ldots, \lambda_n ; K \setminus K_{(l)}, K_{(l)} - \varepsilon) \; .
	\end{equation}
\end{lemma}
\begin{proof}
Observe first that the $l$th front face of $\widehat{\boxempty}_n(\lambda_0, \ldots, \lambda_n; K)$ is given by 
\[ \widehat{F}^0_{K,l}(\lambda(b,c), \lambda(c, d(\lambda))) = (\lambda(b, c- d_{K_{(l)}}(\lambda_{l-1})) e_{K_{(l)}}, \lambda(c-d_{K_{(l)}}(\lambda_{l-1}) e_{K_{(l)}}, d(\lambda))),\]
corresponding to replacing the $d_{K_{(l)}}(\lambda_{l-1})$ box in the diagram \eqref{eq:Fbox-K} by its {lower left edges},
while the back $l$th face is given by taking only the upper right 
edges of the box labeled $d_{K_{(l)}}(\lambda_{l-1})$:
\[ \widehat{F}^1_{K,l}(\lambda(b,c), \lambda(c, d(\lambda))= (\lambda(b + d_{K_{(l)}}(\lambda_{l-1}) e_{K_{(l)}}, c), \lambda(c, d(\lambda))).\]

		To prove the Lemma, we will actually prove a stronger equation: namely,
	\[
		\widehat{F}^\varepsilon_{K, l}( \widehat{\boxempty}_n(\lambda_0, \ldots, \lambda_n; K)) = \widehat{\Xi}(\lambda_0,  \ldots, \lambda_n ; K \setminus K_{(l)}, K_{(l)} - \varepsilon) \; .
	\]
	Writing $J = K \setminus K_{(l)}$, we have 
	\begin{equation}
	\label{eq:J-K}
		J_{(i)} = \begin{cases}
			K_{(i)} \, , & 1 \leq i < l \\
			K_{(i+1)} \, , & l \leq i \leq n-1 
		\end{cases} \; .
	\end{equation}
	In particular, the fact that $K_{(l-1)} < K_{(l)} < K_{(l+1)}$ implies $J_{(l-1)} \leq K_{(l)} - \varepsilon < J_{(l)}$; thus $m := m(J, K_{(l)} - \varepsilon) = l - 1$. 
	Also,  Equation~\eqref{eq:box-map-degree} tells us  that 
	\[
		\widetilde{d}_l (\widehat{\boxempty}_n(\lambda_0, \ldots, \lambda_n; K)) = d_{K_{(l)}}(\lambda_{l-1}) e_{K_{(l)}} \; .
	\]
	By Definition~\ref{def:rectangular-faces} and adopting the notation in Construction~\ref{constr:box-map}, we have
	\begin{align*}
		& \ \widetilde{r} \left( \widehat{F}^0_{K, l}( \widehat{\boxempty}_n(\lambda_0, \ldots, \lambda_n; K)) \right) \\
		& = \widehat{\boxempty}_n(\lambda_0, \ldots, \lambda_n; K) \left( 0 , \widetilde d(\widehat{\boxempty}_n(\lambda_0, \ldots, \lambda_n; K))   \right) \\
		& = \lambda \left(b , d(\lambda) \right) \in \Lambda \;.
	\end{align*}
Similarly, 
\begin{align*}
& \widetilde r  \left( \widehat{F}^1_{K, l}( \widehat{\boxempty}_n(\lambda_0, \ldots, \lambda_n; K)) \right) \\
& = \widehat{\boxempty}_n(\lambda_0, \ldots, \lambda_n; K) \left( d_{K_{(l)}}(\lambda_{l-1}) e_{K_{(l)}} , \widetilde d(\widehat{\boxempty}_n(\lambda_0, \ldots, \lambda_n; K))   \right) \\
&= \lambda \left( b + \widetilde d_l(\widehat \boxempty_n(\lambda_0, \ldots, \lambda_n; K)) , d(\lambda) \right) \in \Lambda.
\end{align*}

	Using the fact that $m = l-1$ and Equation \eqref{eq:J-K}, we now compute that
	\begin{align*}
		b + & \ \widetilde{d}_l (\widehat{\boxempty}_n(\lambda_0, \ldots, \lambda_n; K))  =  \, d_{K_{(l)}}(\lambda_{l-1})  e_{K_{(l)}} + \sum_{i=1}^{n} \sum_{j=K_{(i)}+1 }^{k} d_{j}(\lambda_{i-1}) e_j \\
		& = \sum_{i =1}^{l-1} \sum_{j=K_{(i)}+1}^{k} d_{j}(\lambda_{i-1}) e_j + \sum_{j=K_{(l)}}^k d_j(\lambda_{l-1})  e_j + \sum_{i =l+1}^{n} \sum_{j=K_{(i)}+1}^{k} d_{j}(\lambda_{i-1}) e_j \\
		& = \sum_{i =1}^{m} \sum_{j=J_{(i)}+1}^{k} d_{j}(\lambda_{i-1}) e_j + \sum_{j= (K_{(l)}-1) +1 }^k d_j(\lambda_{m}) e_j + \sum_{i' =m+1}^{n-1} \sum_{j=J_{(i')}+1}^{k} d_{j}(\lambda_{i'}) e_j \; ,
	\end{align*}
	where we have used $i' := i-1$ in the penultimate sum. Since $J = K \backslash K_{(l)}$, we have
\[
		\widetilde{r} \left( \widehat{F}^1_{K, l}( \widehat{\boxempty}_n(\lambda_0, \ldots, \lambda_n; K)) \right) = \widetilde{r} \left( \widehat{\Xi}(\lambda_0,  \ldots, \lambda_n ; K \backslash K_{(l)},  K_{(l)} - 1)  \right) .\]
	Similar computations yield
\[ 		\widetilde{r} \left( \widehat{F}^0_{K, l}( \widehat{\boxempty}_n(\lambda_0, \ldots, \lambda_n; K)) \right) = \widetilde{r} \left( \widehat{\Xi}(\lambda_0,  \ldots, \lambda_n ; K \setminus K_{(l)}, K_{(l)} )  \right) \; \]
	and, for $\varepsilon = 0,1,$	
	\[
		\widetilde{s} \left( \widehat{F}^\varepsilon_{K, l}( \widehat{\boxempty}_n(\lambda_0, \ldots, \lambda_n; K)) \right) = \widetilde{s} \left( \widehat{\Xi}(\lambda_0,  \ldots, \lambda_n ; K \setminus K_{(l)}, K_{(l)} - \varepsilon)  \right) \; .
	\]
	Since $\widetilde \Lambda(v)$ has an initial object, the fact that $\widehat{F}^\varepsilon_{K, l}( \widehat{\boxempty}_n(\lambda_0, \ldots, \lambda_n; K)) $ and $ \widehat{\Xi}(\lambda_0,  \ldots, \lambda_n ; K \setminus K_{(l)}, K_{(l)} - \varepsilon) $ have the same source and range means they must be equal.
\end{proof}

\begin{rmk}
	It follows from Equation \eqref{eq:m-J-q-l} that for any $J \in {k \choose n-1}$, any $q \in \{ 1, \ldots, k \} \backslash J$ and any $\varepsilon \in \{0,1\}$, we have $ J \sqcup \{q\} \in {k \choose n}$ and $q$ is the $(m(J,q) +1)$-th smallest element in $J \sqcup \{q\}$. Hence Equation \eqref{eq:Fbox-K} becomes
	\begin{equation} \label{eq:Fbox-J}
	F^\varepsilon_{J \sqcup q, m(J,q) +1}( \widehat{\boxempty}_n(\lambda_0, \ldots, \lambda_n; J \sqcup \{q\})) = \Xi(\lambda_0,  \ldots, \lambda_n ; J , q - \varepsilon) \; .
	\end{equation}
	Also, for any $J \in {k \choose n-1}$ and any $l \in \{ 1, \ldots, n-1 \}$, direct computation yields
	\begin{align}
	\label{eq:box-0-Xi} \widehat{\boxempty}_{n-1}(\lambda_1, \ldots, \lambda_n; J) = &~ \widehat{\Xi}(\lambda_0,  \ldots, \lambda_n ; J , 0 ) \; ,\\
	\label{eq:box-n-Xi} \widehat{\boxempty}_{n-1}(\lambda_0, \ldots, \lambda_{n-1}\lambda_n; J) = &~ \widehat{\Xi}(\lambda_0,  \ldots, \lambda_n ; J , k ) \; .
	\end{align}
	Diagrammatically, both sides in \eqref{eq:box-0-Xi} are represented by 
	\begin{equation*}
	\begin{array}{ccccc}
	\cline{1-1} 
	\multicolumn{1}{c|}{d_{1}(\lambda_{0})} & d_{1}(\lambda_{1}) & \cdots & {d_{1}(\lambda_{n-1})}  & d_{1}(\lambda_{n}) \\ 
	\multicolumn{1}{c|}{\vdots} & \vdots & \cdots  & \vdots  & \vdots \\
	\cline{2-2} 
	\vdots & \multicolumn{1}{|c|}{d_{J_{(1)}}(\lambda_{1})} & \cdots  & \vdots  & \vdots \\
	\cline{2-2} 
	\vdots & \multicolumn{1}{c|}{\vdots} & \cdots  & \vdots  & \vdots \\
	\vdots & \vdots &  {\cdots} & \vdots  & \vdots \\
	\vdots & \vdots &  \multicolumn{1}{c|}{\cdots} & {\vdots}  & \vdots \\
	\cline{4-4} 
	\vdots & \vdots & \cdots  & \multicolumn{1}{|c|}{d_{J_{(n-1)}}(\lambda_{n-1})}  &   \vdots \\
	\cline{4-4} 
	\vdots & \vdots & \cdots  & \multicolumn{1}{c|}{\vdots}   & \vdots \\
	d_{k}(\lambda_{0}) & d_{k}(\lambda_{1}) & \cdots & d_{k}(\lambda_{n-1})   & \multicolumn{1}{|c}{  d_{k}(\lambda_{n})} 
	\end{array}
	\label{eq:diagram-box}
	\end{equation*}
	and both sides in \eqref{eq:box-n-Xi} are represented by
	\begin{equation*}
	\begin{array}{ccccc}
	\multicolumn{1}{|c}{d_{1}(\lambda_{0})} & \cdots & d_{1}(\lambda_{n-2}) & {d_{1}(\lambda_{n-1})}  & d_{1}(\lambda_{n}) \\
	\multicolumn{1}{|c}{\vdots} & \cdots & \vdots  & \vdots  & \vdots \\
	\cline{1-1} 
	\multicolumn{1}{|c|}{d_{J_{(1)}}(\lambda_{0})} & \cdots & \vdots  & \vdots  & \vdots \\
	\cline{1-1} 
	\multicolumn{1}{c|}{\vdots} & \cdots & \vdots  & \vdots  & \vdots \\
	\vdots &  {\cdots} & \vdots & \vdots  & \vdots \\
	\vdots & \multicolumn{1}{c|}{\cdots} & \vdots &  {\vdots}  & \vdots \\
	\cline{3-3} 
	\vdots & \cdots & \multicolumn{1}{|c|}{d_{J_{(n-1)}}(\lambda_{n-2})}  &  \vdots  &  \vdots \\
	\cline{3-3} 
	\vdots & \cdots & \multicolumn{1}{c|}{\vdots}   & \vdots  & \vdots \\
	d_{k}(\lambda_{0}) & \cdots & d_{k}(\lambda_{n-2}) & \multicolumn{1}{|c}{ d_{k}(\lambda_{n-1}) }  &  d_{k}(\lambda_{n}) \\
	\cline{4-5} 
	\end{array}
	\label{eq:diagram-box}
	\end{equation*}
\end{rmk}

\begin{lemma}
	For any $J \in {k \choose n-1} $ and any $l \in \{ 1, \ldots, n-1\}$, 
	\begin{equation}
	\label{eq:box-l-Xi} \Box_{n-1}(\lambda_0, \ldots, \lambda_{l-1}\lambda_{l}, \ldots, \lambda_n; J) = \sum_{\varepsilon = 0}^1 \Xi(\lambda_0,  \ldots, \lambda_n ; J , J_{(l)} - \varepsilon) .
	\end{equation}
	\label{lem:8-17}
\end{lemma}
\begin{proof} 
	Write $\vec{\lambda} = (\lambda_0, \ldots, \lambda_n)$.
	We first observe, via a computation similar to those in the proof of Lemma~\ref{lem:rectangular-bdry}, that 
	\begin{align*}\widehat{F}^1_{J,l}\left(\widehat{\Xi}\left(\vec{\lambda}; J, J_{(l)}\right)\right) &=
	\widehat{F}^0_{J,l} \left(\widehat{\Xi}\left(\vec{\lambda}; J, J_{(l)}-1 \right)\right) \; .
	\end{align*}
	Diagrammatically, both are represented by
	
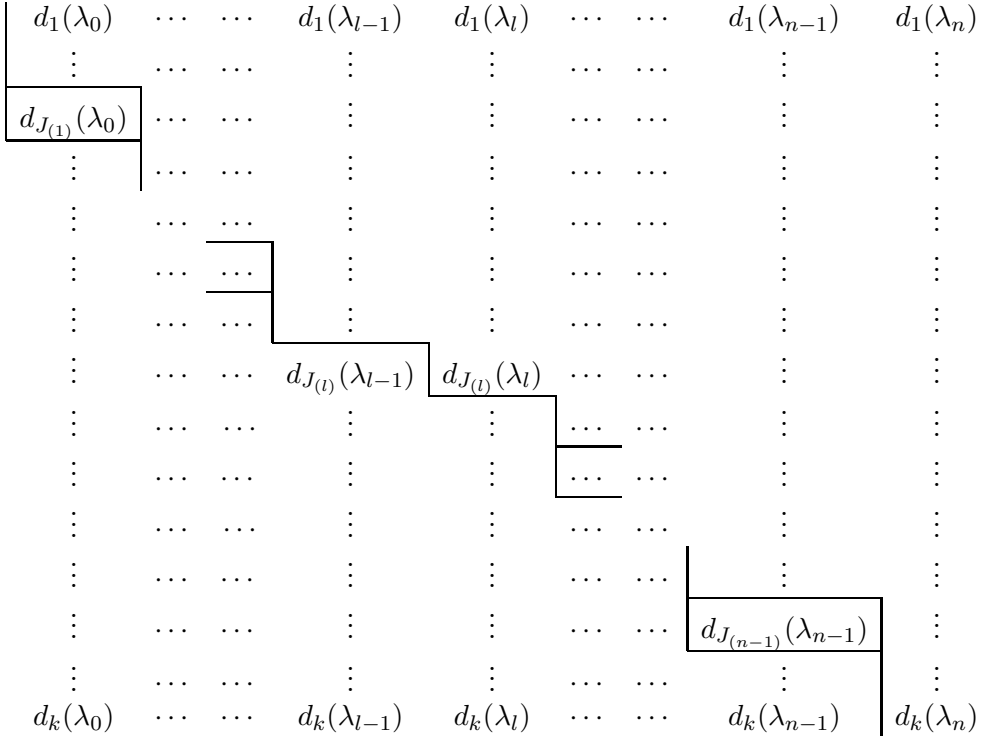
\[
	\begin{array}{ccccccccc}
	\multicolumn{1}{|c}{d_{1}(\lambda_{0})} & \cdots & \cdots & d_{1}(\lambda_{l-1}) & d_{1}(\lambda_{l}) & \cdots & \cdots & {d_{1}(\lambda_{n-1})}  & d_{1}(\lambda_{n}) \\
	\multicolumn{1}{|c}{\vdots} & \cdots & \cdots & \vdots & \vdots  & \cdots & \cdots & \vdots  & \vdots \\
	\cline{1-1} 
	\multicolumn{1}{|c|}{d_{J_{(1)}}(\lambda_{0})} & \cdots & \cdots & \vdots & \vdots  & \cdots & \cdots & \vdots  & \vdots \\
	\cline{1-1} 
	\multicolumn{1}{c|}{\vdots} & \cdots & \cdots & \vdots & \vdots  & \cdots & \cdots & \vdots  & \vdots \\
	\vdots & \cdots & \cdots & \vdots & \vdots  & \cdots & \cdots & \vdots  & \vdots \\
	\cline{3-3} 
	\vdots & \cdots & \multicolumn{1}{c|}{\cdots} & \vdots & \vdots  & \cdots & \cdots & \vdots  & \vdots \\
	\cline{3-3} 
	\vdots & \cdots & \multicolumn{1}{c|}{\cdots} & \vdots & \vdots  & \cdots & \cdots & \vdots  & \vdots \\
	\cline{4-4} 
	\vdots & \cdots & \cdots  & \multicolumn{1}{c|}{d_{J_{(l)}}(\lambda_{l-1})} & {d_{J_{(l)}}(\lambda_{l})} & \cdots & \cdots & \vdots  & \vdots \\
	\cline{5-5} 
	\vdots & \cdots & {\cdots} & \vdots & \multicolumn{1}{c|}{\vdots} &  \cdots & \cdots & \vdots  & \vdots \\
	\cline{6-6} 
	\vdots & \cdots & \cdots  & \vdots &  \vdots & \multicolumn{1}{|c}{\cdots} & \cdots & \vdots  & \vdots \\
	\cline{6-6} 
	\vdots & \cdots & {\cdots} & \vdots &  \vdots & \cdots & \cdots & \vdots  & \vdots \\
	\vdots  & \cdots & \cdots  & \vdots & \vdots & \cdots & \cdots  & \multicolumn{1}{|c}{\vdots}  & \vdots \\
	\cline{8-8} 
	\vdots  & \cdots & \cdots  & \vdots & \vdots & \cdots & \cdots  & \multicolumn{1}{|c|}{d_{J_{(n-1)}}(\lambda_{n-1})}  &   \vdots \\
	\cline{8-8} 
	\vdots  & \cdots & \cdots  & \vdots & \vdots & \cdots & \cdots  & \multicolumn{1}{c|}{\vdots}   & \vdots \\
	d_{k}(\lambda_{0}) & \cdots & \cdots & d_{k}(\lambda_{l-1}) & d_{k}(\lambda_{l}) & \cdots & \cdots & d_{k}(\lambda_{n-1})   & \multicolumn{1}{|c}{d_{k}(\lambda_{n})} 
	\end{array}
	\]
	The rule for adding $n$-rectangles, Lemma~\ref{lem:rectangular-bdry-concat}, then implies that 
	\[ 
		\Xi\left(\vec{\lambda}; J, J_{(l)}-1\right) + \Xi\left(\vec{\lambda}; J, J_{(l)}\right)= \rectangularchain{ R_{J,l}\left(\widehat{\Xi}\left(\vec{\lambda}; J, J_{(l)}\right)\right) \, \widehat{\Xi}\left(\vec{\lambda}; J, J_{(l)}-1 \right) }_{J} \; .
	\]
	Diagrammatically, both sides of the above equation are represented by replacing 
	\[
		\begin{array}{cc}
			\cline{1-1} 
			\multicolumn{1}{c|}{d_{J_{(l)}}(\lambda_{l-1})} & {d_{J_{(l)}}(\lambda_{l})}  \\
			\cline{2-2} 
		\end{array}
	\]
	in the above diagram by 
	\[
	\begin{array}{cc}
		\cline{1-2} 
		\multicolumn{1}{|c}{d_{J_{(l)}}(\lambda_{l-1})} & \multicolumn{1}{c|}{d_{J_{(l)}}(\lambda_{l})}  \\
		\cline{1-2} 
	\end{array} \; .
	\]
	Comparing with \eqref{eq:Fbox-K}, we see that 
	\[ R_{J,l}\left(\widehat{\Xi}\left(\vec{\lambda}; J, J_{(l)}\right)\right) \, \widehat{\Xi}\left(\vec{\lambda}; J, J_{(l)}-1 \right)  = \widehat{\boxempty}_{n-1}(\lambda_0, \ldots, \lambda_{l-1} \lambda_l, \ldots, \lambda_n; J) \; . \]
	This, combined with the last equation, proves the claim. 
\end{proof}

\begin{rmk}
	Equation \eqref{eq:m-J-q-l} implies that  for any $J \in {k \choose n-1}$, any $l \in \{ 1, \ldots, n-1 \}$ and any 
	 $q \in J$, we have $q = J_{(l)}$ if and only if $l = m(J,q)$. Hence whenever $q \in J$, Equation \eqref{eq:box-l-Xi} above can be equivalently written as
	\begin{equation} \label{eq:box-q-Xi} 
		\Box_{n-1}(\lambda_0, \ldots, \lambda_{m(J,q)-1}\lambda_{m(J,q)}, \ldots, \lambda_n; J) = \sum_{\varepsilon = 0}^1 \Xi(\lambda_0,  \ldots, \lambda_n ; J , q - \varepsilon) \; .
	\end{equation}
\end{rmk}

We now begin the proof of Theorem~\ref{thm:box-is-chain-map}, that is, showing $\partial_n \circ \boxempty_n  = \boxempty_{n-1} \circ \partial_n^{\Presolution}$. 
\begin{proof}[Proof of Theorem~\ref{thm:box-is-chain-map}]
		The statement is obvious for $n < 0$. The case for $n = 0$ follows from the observation that $\boxempty_0(\lambda) = (r(\lambda), \lambda)$ for any $\lambda \in \Lambda^{*(0+1)}$. It therefore suffices to show that whenever $n >0$, 
	\[
		\left(\partial_n \circ \boxempty_n\right)_v (\lambda_0,  \ldots, \lambda_n)  = \left(\boxempty_{n-1} \circ \partial_n^{\Presolution}\right)_v (\lambda_0,  \ldots, \lambda_n)  \; .
	\]
	For any $J \in {k \choose n-1}$, we have the following equations. For the sake of brevity, we write $\vec{\lambda}$ in place of $(\lambda_0,  \ldots, \lambda_n)$, drop the subscript $v$, and provide justification for most steps via the equation references above the equation signs. 
	\begin{align*}
 \Box_{n-1}(\lambda_1, \ldots, \lambda_n; J)  &+ (-1)^n \Box_{n-1}(\lambda_0, \ldots, \lambda_{n-1}\lambda_n; J)  \\
	\overset{\eqref{eq:box-0-Xi}, \eqref{eq:box-n-Xi}}{\hfill =} &~ \Xi(\vec{\lambda} ; J , 0 )  + (-1)^n \,\Xi(\vec{\lambda} ; J , k ) \\
	\overset{\eqref{eq:m-J-0}, \eqref{eq:m-J-k}}{\hfill =}  &~ (-1)^{m(J,0)} \,\Xi(\vec{\lambda} ; J , 0 )  - (-1)^{m(J,k)} \,\Xi(\vec{\lambda} ; J , k ) \\
	\overset{}{\hfill =} &~ \sum_{q=0}^{k-1} (-1)^{m(J,q)} \,\Xi(\vec{\lambda} ; J , q) - \sum_{q=1}^{k} (-1)^{m(J,q)} \,\Xi(\vec{\lambda} ; J , q) \\
	\overset{}{\hfill =} &~ \sum_{q=1}^{k} \left( (-1)^{m(J,q-1)} \,\Xi(\vec{\lambda} ; J , q-1) - (-1)^{m(J,q)} \,\Xi(\vec{\lambda} ; J , q) \right) \\
	\overset{\eqref{eq:m-J-q}}{\hfill =} &~  \sum_{q \in \{1, \ldots, k\} \backslash J } \left( (-1)^{m(J,q)} \,\Xi(\vec{\lambda} ; J , q-1) - (-1)^{m(J,q)} \,\Xi(\vec{\lambda} ; J , q) \right)  \\
	& + \sum_{q \in J } \left( (-1)^{m(J,q) -1} \,\Xi(\vec{\lambda} ; J , q-1) - (-1)^{m(J,q)} \,\Xi(\vec{\lambda} ; J , q) \right) \\
	\overset{}{\hfill =} &~ \sum_{q \in \{1, \ldots, k\} \backslash J } (-1)^{m(J,q) + 1} \left( \Xi(\vec{\lambda} ; J , q) - \Xi(\vec{\lambda} ; J , q - 1) \right)	\\
	& - \sum_{q \in J } (-1)^{m(J,q)} \left( \Xi(\vec{\lambda} ; J , q) + \Xi(\vec{\lambda} ; J , q-1) \right) .
	\\
	\overset{\eqref{eq:Fbox-J}, \eqref{eq:box-q-Xi}}{\hfill =} &~  \sum_{q \in \{1, \ldots, k\} \backslash J } (-1)^{m(J,q) + 1} \sum_{\varepsilon = 0}^{1} (-1)^{\varepsilon} {F^\varepsilon_{m(J,q)+1}( \widehat{\boxempty}_n(\vec{\lambda}; J \sqcup \{q\}))}  \\ 
	& - \sum_{q \in J } (-1)^{m(J,q)} \Box_{n-1}(\lambda_0, \ldots, \lambda_{m(J,q)-1}\lambda_{m(J,q)}, \ldots, \lambda_n; J) 
	\end{align*}
	Rearranging the final entries, we obtain for any $J \in {k \choose n-1}$, 
	\begin{align*}
	\boxempty_{n-1} ( \partial_n^{\Presolution}( \vec{\lambda} ); J) 
	= &~ \Box_{n-1}(\lambda_1, \ldots, \lambda_n; J) + (-1)^n \Box_{n-1}(\lambda_0, \ldots, \lambda_{n-1}\lambda_n; J)  \\
	& + \sum_{l = 1 }^{n-1} (-1)^{l} \, \Box_{n-1}(\lambda_0, \ldots, \lambda_{l-1}\lambda_{l}, \ldots, \lambda_n; J) \\
	= &~ \sum_{q \in \{1, \ldots, k\} \backslash J } \sum_{\varepsilon = 0}^{1} (-1)^{m(J,q) + 1 + \varepsilon} {F^\varepsilon_{m(J,q)+1}( \widehat{\boxempty}_n(\vec{\lambda}; J \sqcup \{q\}))}   
	\end{align*}
	Taking the sum over all $J \in {k \choose n-1}$ and noting that for each pair $(J,q)$ where $q \in \{1, \ldots, k\} \backslash J$, there is a unique pair $(K,l) \in {k \choose n} \times \{1, \ldots, n \}$ such that $K = J \sqcup \{q\}$, $q = K_{(l)} $, and $  J_{(l-1)} < q < J_{(l)} $, we obtain
	\begin{align*}
	\boxempty_{n-1} ( \partial_n^{\Presolution} (\vec{\lambda} )) 
	= &~ \sum_{J \in {k \choose n-1}} \boxempty_{n-1} ( \partial_n^{\Presolution}( \vec{\lambda}) ; J) \\
	= &~ \sum_{J \in {k \choose n-1}} \sum_{q \in \{1, \ldots, k\} \backslash J } \sum_{\varepsilon = 0}^{1} (-1)^{m(J,q) + 1 + \varepsilon} {F^\varepsilon_{m(J,q)+1}( \widehat{\boxempty}_n(\vec{\lambda}; J \sqcup \{q\}))}   \\
	\overset{\eqref{eq:m-J-q-l}}{\hfill =}  &~ \sum_{K \in {k \choose n}} \sum_{l = 1 }^n \sum_{\varepsilon = 0}^{1} (-1)^{l + \varepsilon} {F^\varepsilon_{l}( \widehat{\boxempty}_n(\vec{\lambda}; K ))} \\
	= &~ \sum_{K \in {k \choose n}} \partial_n ( \Box_n(\vec{\lambda}; K )) \\
	= &~ \partial_n ( \Box_n(\vec{\lambda})) \; ,
	\end{align*}
	which is what we wanted. 
\end{proof}

\subsection{A proof of isomorphism by chain maps}

We are ready to complete our second proof of the isomorphism between the cubical and categorical \mbox{(co-)}homology groups of a $k$-graph $\Lambda$. 

\begin{prop}\label{prop:homotopy-equivalence}
	For any $k$-graph $\Lambda$, the $\Lambda$-chain maps $\triangledown_*$ and $\boxempty_*$ implement a homotopy equivalence between the $\Lambda$-chain complexes $\Presolution_*(\Lambda) $ and $ \Z {\widetilde{Q}}_*( \widetilde{\Lambda})$. 
\end{prop}

\begin{proof}
	We observe from Construction~\ref{constr:box-map} that for any $v \in \operatorname{Obj}\Lambda$, any $\left( \lambda_0, \ldots, \lambda_n \right) \in \Lambda^{*(n+1)} v$, and any $K \in {k \choose n}$, if $d_{K_{(i)}}(\lambda_{i-1})=0$ for some $i$, then the color type $C \left(\widehat{\boxempty}_n(\lambda_0, \ldots, \lambda_n; K)\right)$ is properly contained in $K$, and thus the resulting rectangular chain $\left({\boxempty}_n\right)_v (\lambda_0, \ldots, \lambda_n; K)$ is trivial. In particular, using the notation of Construction~\ref{constr:triangle-map}, we see that for any $v \in \operatorname{Obj}\Lambda$ and any $(\lambda, \mu) \in Q_n(\widetilde{\Lambda}(v))$, we have $\left({\boxempty}_n\right)_v (\lambda_1^\sigma, \ldots, \lambda_n^\sigma, \mu; K) =0$ for any \emph{nontrivial} permutation $\sigma \in \Sigma_{n}$ and for any $K \in {k \choose n}$ except for $K = C(\lambda)$. This implies 
	\[
		\left(\Box_n \circ \triangledown_n\right)_v (\lambda, \mu) = \left(\Box_n\right)_v \big(\lambda_1^e, \ldots, \lambda_n^e, \mu; C(\lambda) \big) = (\lambda, \mu) \;,
	\]
	where $e$ denotes the identity in $\Sigma_n$. Hence we have
	\begin{equation}
		\Box_n \circ \triangledown_n = \operatorname{id}_{\Z {\widetilde{Q}}_*( \widetilde{\Lambda})} \;.
		\label{eq:box-circ-triangle-id}
	\end{equation}
	
	On the other hand, since Remark~\ref{rmk:categorical-resolution} implies that $\chaincomplex{\mathcal P_n, \partial_n^{\Presolution} }_n$ is a projective resolution of the trivial $\Lambda$-module $\Z^\Lambda$, \cite[Theorem~III.6.1]{maclane} implies that $\triangledown_n \circ \Box_n$ is chain homotopic to the identity chain map on $\Presolution_*(\Lambda) $. This completes the proof. 
\end{proof}

\begin{rmk}
	The proof of Proposition~\ref{prop:homotopy-equivalence} does not use any result from Section~\ref{sec:initial-vertices} and in particular does not assume the knowledge that $\Z {\widetilde{Q}}_*( \widetilde{\Lambda})$ is exact. It thus gives an alternative proof that Equation~\eqref{eq:cubical-free-res} is a projective resolution. Conversely, if we assume knowledge of Theorem~\ref{thm:proj-res}, then Proposition~\ref{prop:homotopy-equivalence} follows immediately, as all chain maps between projective resolutions are homotopy equivalences (see \cite[Theorem III.6.1]{maclane}). To summarize, the contents of this section up to this point can be seen as parallel to Section~\ref{sec:initial-vertices} and Theorem~\ref{thm:proj-res}. 
\end{rmk}

The advantage of this alternative proof is that it provides explicit chain maps that induce isomorphisms between cubical and categorical \mbox{(co-)}homology groups. More precisely, the pair of $\Lambda$-chain maps $\triangledown_*$ and $\boxempty_*$ induce chain maps between categorical and cubical \mbox{(co-)}chain complexes (cf.~Constructions~\ref{constr:cubical-homology-pre} and ~\ref{constr:cubical-cohomology-pre}, and Remark~\ref{rmk:categorical-cohlogy-reformulations}). 

\begin{thm}\label{thm:chain-maps}
	For any $k$-graph $\Lambda$ and any right $\Lambda$-module $\mathcal M$, there is a pair of cochain maps
	\[
		\triangledown^*_{\mathcal{M}} \colon C^*_{} ( \Lambda, \mathcal{M}) \to C^*_{\operatorname{cub}} ( \Lambda, \mathcal{M}) \quad \text{ and } \quad \boxempty^*_{\mathcal{M}} \colon C^*_{\operatorname{cub}} ( \Lambda, \mathcal{M}) \to C^*_{} ( \Lambda, \mathcal{M})
	\]
	defined by 
	\[
		\triangledown^n_{\mathcal{M}} (f) (\eta) =  \sum_{\sigma \in \Sigma_{n}} \operatorname{sgn}(\sigma) \; f\left( \eta_1^\sigma, \ldots, \eta_n^\sigma \right) 
	\]
	for any $n \in \Z$, any $f \in C^n_{} ( \Lambda, \mathcal{M})$ and any $\eta \in Q_n(\Lambda)$, and 
	\[
		\boxempty^n_{\mathcal{M}} (g) \left( \lambda_0, \ldots, \lambda_{n-1} \right) = \sum_{ { K \in {k \choose n} } } \sum_{m \in [b + e_K, c] \cap \N^k} g \Big( \lambda (m - e_K , m) \Big) \cdot \lambda \big( m, d \left( \lambda \right) \big) 
	\] 
	for any $n \in \Z$, any $g \in C^n_{\operatorname{cub}} ( \Lambda, \mathcal{M})$ and any $\left( \lambda_0, \ldots, \lambda_{n-1} \right) \in \Lambda^{*n}$, where $e_K = \sum_{j \in K}{e_j} \in \N^k$, $\lambda = \lambda_0 \cdots \lambda_{n-1}$, 
	\[
		b =  \sum_{i=1}^{n} \sum_{j=K_{(i)}+1 }^{k} d_{j}(\lambda_{i-1}) e_j \quad\text{ and }\quad c = \sum_{i=1}^{n} \sum_{j= K_{(i)}}^{k}  d_{j}(\lambda_{i-1}) e_j .
	\]
	Moreover, they constitute a homotopy equivalence and thus also induce an isomorphism between $H^n(\Lambda, \mathcal M)$ and $H^n_{\operatorname{cub}}(\Lambda, \mathcal M)$. 
	
	Similarly, for any left $\Lambda$-module $\mathcal N$, there is a pair of chain maps
	\[
		\triangledown_*^{\mathcal{N}} \colon C_*^{\operatorname{cub}} ( \Lambda, \mathcal{N}) \to C_*^{} ( \Lambda, \mathcal{N})  \quad \text{ and } \quad 	\boxempty_*^{\mathcal{N}} \colon C_n^{} ( \Lambda, \mathcal{N}) \to C_n^{\operatorname{cub}} ( \Lambda, \mathcal{N}) 
	\]
	defined by 
	\[
		\triangledown_n^{\mathcal{N}} \left( \eta , a \right) = \sum_{\sigma \in \Sigma_{n}} \operatorname{sgn}(\sigma)  \left( (\eta_1^\sigma, \ldots, \eta_n^\sigma)  , a\right)
	\]
	for any $n \in \Z$, any $\eta \in Q_n(\Lambda)$ and any $a \in \mathcal{N}(s(\eta))$, and 
	\[
		\boxempty_n^{\mathcal{N}} \big( \left( \lambda_0, \ldots, \lambda_{n-1} \right) , a \big) = \sum_{ { K \in {k \choose n} } } \sum_{m \in [b + e_K, c] \cap \N^k} \Big( \lambda (m-e_K, m) ,   \lambda \big( m, d \left( \lambda \right) \big)  \cdot a  \Big)
	\]
	for any $n \in \Z$, any $\left( \lambda_0, \ldots, \lambda_{n-1} \right) \in \Lambda^{*n}$ and any $a \in \mathcal{N}(s(\lambda_{n-1}))$, using the same notations as above. Moreover, they constitute a homotopy equivalence, and thus also induce an isomorphism between $H_n(\Lambda, \mathcal N)$ and $H_n^{\operatorname{cub}}(\Lambda, \mathcal N)$. 
\end{thm}

\begin{proof}
	This follows directly from Propositions~\ref{prop:homotopy-equivalence} and~\ref{prop:cubical-chain-iso}, after untangling the definitions. 
\end{proof}

\subsection{Naturality}

We establish naturality of our chain maps. In the following, we use $\circ$ to denote a composition of two natural transformations, and juxtaposition to denote a composition of two functors or between a natural transformation and a functor. 

\begin{lemma}
	\label{lem:naturality-equivariant}
	For any $k$-graph morphism $\varphi \colon  \Lambda \to \Lambda'$, if we let $\widetilde{\varphi}$ be the natural transformation as in Remark~\ref{rmk:future-path-graph}\eqref{rmk:future-path-graph:natural-morphism}, let $\Z \widetilde{Q}_*(\widetilde{\varphi})$ denote the induced $\Lambda$-module chain map from $\Z \widetilde{Q}_*(\widetilde{\Lambda})$ to $\Z \widetilde{Q}_*(\widetilde{\Lambda'})  \varphi$ and let $\Presolution_*({\varphi})$ denote the induced $\Lambda$-module chain map from $\Presolution_*(\Lambda)$ to $\Presolution_*(\Lambda')  \varphi$, then we have 
	\[
		\Presolution_*({\varphi}) \circ \triangledown_{*,\Lambda} = \left(\triangledown_{*,\Lambda'} \, \varphi \right) \circ \Z \widetilde{Q}_*(\widetilde{\varphi}) \quad \text{and} \quad \Z \widetilde{Q}_*(\widetilde{\varphi}) \circ \boxempty_{*,\Lambda} = \left(\boxempty_{*,\Lambda'} \, \varphi \right) \circ \Presolution_*({\varphi}) \; .
	\]
\end{lemma}
\begin{proof}
	For any $v \in \operatorname{Obj}\Lambda$, any $n \in \N$ and any $(\lambda, \mu) \in Q_n(\widetilde{\Lambda}(v))$, we have $\widetilde d(\widetilde \varphi_v (\lambda, \mu )) = d(\varphi(\lambda)) = d(\lambda) = \widetilde d(\lambda, \mu)$ since $\varphi$ is degree-preserving. Similarly, for any $\sigma \in \Sigma_{n}$, the multiplicativity of $\varphi$ implies that
	\[\varphi(\lambda) = \varphi(\lambda_1^\sigma) \cdots \varphi(\lambda_n^\sigma).\]  Since 
	$d(\lambda_i^\sigma) = d(\varphi(\lambda_i^\sigma)) = e_{\sigma(i)}$, the factorization property tells us that
	\[\varphi(\lambda)_i^\sigma = \varphi(\lambda_i^\sigma ) \]
	for any $i \in \{1, \ldots, n\}$ and any $\sigma \in \Sigma_{n}$.  This implies
	\[
		\left(\left(\triangledown_{n,\Lambda'} \, \varphi \right) \circ \Z \widetilde{Q}_n(\widetilde{\varphi})\right)_v (\lambda, \mu) = \left( \Presolution_n({\varphi}) \circ \triangledown_{n,\Lambda} \right)_v (\lambda, \mu)  \; .
	\]
	Similarly, the second equality follows from the observation that 
	\[
		\widehat{\boxempty}_n \left(\varphi(\lambda_0), \ldots, \varphi(\lambda_n); K \right) = \varphi \left( \widehat{\boxempty}_n \left(\lambda_0, \ldots, \lambda_n; K \right)\right) 
	\]
	for any $\left( \lambda_0, \ldots, \lambda_{n} \right) \in \Lambda^{*(n+1)}$ and $K \in {k \choose n}$. 
\end{proof}

This lemma leads to the naturality of the \mbox{(co-)}chain maps in Theorem~\ref{thm:chain-maps}. First we specify what type of naturality is considered. 

Let $k \text{\textendash} \mathfrak{graph} \text{\textendash} \mathfrak{LMod}$ be the category such that its objects are pairs $(\Lambda, \mathcal{N})$, where $\Lambda$ is a $k$-graph and $\mathcal{N}$ is a left $\Lambda$-module, and a morphism from $(\Lambda, \mathcal{N})$ to $(\Lambda', \mathcal{N}')$ is a pair $(\varphi, \psi)$, where $\varphi \colon \Lambda \to \Lambda'$ is a $k$-graph morphism and $\psi \colon \mathcal{N} \to \mathcal{N}' \varphi$ is a $\Lambda$-module map. Then $C_*^{} ( -, -)$ and $C_*^{\operatorname{cub}} (-,-)$ form functors from $k \text{\textendash} \mathfrak{graph} \text{\textendash} \mathfrak{LMod}$ to $\mathfrak{chain}$ upon defining 
\begin{align*}
	C_n^{} ( \varphi, \psi) \colon \quad C_n^{} ( \Lambda, \mathcal{N})  & \to  C_n^{} ( \Lambda', \mathcal{N}') \\
	((\lambda_0, \ldots, \lambda_{n-1}) , a)  & \mapsto \left(  \left( \varphi(\lambda_0), \ldots, \varphi(\lambda_{n-1}) \right) , \psi(a) \right)
\end{align*}
and 
\begin{align*}
	C_n^{\operatorname{cub}} ( \varphi, \psi) \colon \quad C_n^{\operatorname{cub}} ( \Lambda, \mathcal{N})  & \to  C_n^{\operatorname{cub}} ( \Lambda', \mathcal{N}') \\
	(\lambda , a) & \mapsto (\varphi(\lambda)  , \psi(a)) \; .
\end{align*}

Similarly, let $k \text{\textendash} \mathfrak{graph}^{\operatorname{op}} \text{\textendash} \mathfrak{RMod}$ be the category such that its objects are pairs $(\Lambda, \mathcal{M})$, where $\Lambda$ is a $k$-graph and $\mathcal{M}$ is a right $\Lambda$-module, and a morphism from $(\Lambda, \mathcal{M})$ to $(\Lambda', \mathcal{M}')$ is a pair $(\varphi, \psi)$, where $\varphi \colon \Lambda' \to \Lambda$ is a $k$-graph morphism and $\psi \colon \mathcal{M} \varphi \to \mathcal{M}'$ is a $\Lambda'$-module map. Then $C^*_{} ( -, -)$ and $C^*_{\operatorname{cub}} (-,-)$ form functors from $k \text{\textendash} \mathfrak{graph}^{\operatorname{op}} \text{\textendash} \mathfrak{RMod}$ to $\mathfrak{cochain}$ upon defining 
\begin{align*}
	C^n_{} ( \varphi, \psi) (f) (\lambda_0, \ldots, \lambda_{n-1}) &= \psi \left( f \left( \varphi(\lambda_0), \ldots, \varphi(\lambda_{n-1}) \right) \right)\\
	C^n_{\operatorname{cub}} ( \varphi, \psi) (g) (\lambda) &= \psi \left( g \left( \varphi(\lambda) \right) \right) 
\end{align*}
for any $n \in \Z$, any $f \in C^n_{} ( \Lambda, \mathcal{M})$, any $g \in C^n_{\operatorname{cub}} ( \Lambda, \mathcal{M})$, any $\left( \lambda_0, \ldots, \lambda_{n-1} \right) \in \Lambda^{*n}$ and any $\lambda \in Q_n(\Lambda)$. 

\begin{prop}\label{prop:naturality}
	For any $k$-graph $\Lambda$, let $\boxempty_{*,\Lambda}^\mathcal{N}$, $\triangledown_{*,\Lambda}^\mathcal{N}$, $\boxempty^{*,\Lambda}_\mathcal{M}$, and $\triangledown^{*,\Lambda}_\mathcal{M}$ be the $\Lambda$-\mbox{(co-)}chain maps as defined in Theorem~\ref{thm:chain-maps}, where the dependence on $\Lambda$ is emphasized in the notations. Then
	\begin{enumerate}
		\item the collections $\left\{ \boxempty_{*,\Lambda}^\mathcal{N} \right\}$ and $\left\{ \triangledown_{*,\Lambda}^\mathcal{N} \right\}$, where $(\Lambda, \mathcal{N})$ ranges over the objects of $k \text{\textendash} \mathfrak{graph} \text{\textendash} \mathfrak{LMod}$, form a pair of natural transformations between $C_*^{} ( -, -)$ and $C_*^{\operatorname{cub}} (-,-)$, and thus induce natural isomorphisms between the homology functors $H_*^{} ( -, -)$ and $H_*^{\operatorname{cub}} (-,-)$; 
		\item the collections $\left\{ \boxempty^{*,\Lambda}_\mathcal{M} \right\}$ and $\left\{ \triangledown^{*,\Lambda}_\mathcal{M} \right\}$, where $(\Lambda, \mathcal{M})$ ranges over the objects of $k \text{\textendash} \mathfrak{graph}^{\operatorname{op}} \text{\textendash} \mathfrak{RMod}$, form a pair of natural transformations between $C^*_{} ( -, -)$ and $C^*_{\operatorname{cub}} (-,-)$, and thus induce natural isomorphisms between the cohomology functors $H^*_{} ( -, -)$ and $H^*_{\operatorname{cub}} (-,-)$.
	\end{enumerate}
\end{prop}

\begin{proof}
	This follows from Lemma~\ref{lem:naturality-equivariant} and Proposition~\ref{prop:cubical-chain-iso}. 
\end{proof}

\subsection{Cochain maps in low dimensions}

In \cite{kps-twisted}, Kumjian, Pask, and Sims constructed explicit maps on the cocycles that induce isomorphisms between cubical and categorical cohomology groups in low dimensions ($n \leq 2$) and with constant coefficients. Fixing a constant right $\Lambda$-module $M^{\Lambda}$ associated to an abelian group $M$, we now compare our cochain maps $\boxempty^{n}_{M^{\Lambda}} $ and $\triangledown^{n}_{M^{\Lambda}} $ (see Theorem~\ref{thm:chain-maps}) with the maps defined in \cite{kps-twisted}, for $n = 0$, $1$, and $2$. For the sake of brevity, we write $M$ in place of $M^{\Lambda}$. 

For $n=0$, identifying both 0-cubes and composable 0-tuples with vertices, we see that both $\boxempty_{0} $ and $\triangledown_{0} $ induce the identity map on vertices. Hence $\boxempty^{0}_{M} $ and $\triangledown^{0}_{M} $ give the identity map on $M$-valued functions over the vertices, which agrees with the construction in \cite[Remark~3.9]{kps-twisted}. 

When $n=1$, the recipe from \cite{kps-twisted} for passing from a  composable 1-tuple  $\lambda \in \Lambda$ to a linear combination of 1-cubes can be found in Theorem 3.10 and in particular Equation (3.7) of \cite{kps-twisted}.  To explain the recipe, we decompose an arbitrary $\lambda \in \Lambda$ as a product $\underline{\lambda}_k \underline{\lambda}_{k-1} \cdots \underline{\lambda}_1$, where $d(\underline{\lambda}_i)  = d_i(\lambda) e_{i} $ for all $i$, and then write $\underline{\lambda}_i = \underline{\lambda}_{i,1} \ldots \underline{\lambda}_{i,d_i(\lambda)}$ where $\underline\lambda_{i,j}$ is an edge of degree $e_i$ for all $i$ and $j$. Given a cubical 1-cocycle $f \in  Z^1_{\operatorname{cub}} ( \Lambda, M) \subset C^1_{\operatorname{cub}} ( \Lambda, M)$, Kumjian, Pask, and Sims define a categorical 1-cocycle $\widetilde{f} \in Z^1_{} ( \Lambda, M)$ by 
\[
\widetilde{f} (\lambda) = \sum_{i=1}^{k} \sum_{j = 1}^{d_i(\lambda)} f(\underline{\lambda}_{i,j})
\]
In fact, \cite[Theorem 3.10]{kps-twisted} establishes that $\widetilde{f} (\lambda) = \sum_{i=1}^{r} f({\lambda}_{r})$ for any sequence $\lambda_1 , \ldots, \lambda_r$ of edges with $\lambda_1 \cdots \lambda_r  = \lambda$. It is straightforward to check that $\widetilde{f} = \boxempty^{1}_{M} (f)$ by identifying ${k \choose 1}$ with $\{ 1, \ldots, k\}$ and $\underline{\lambda}_{i,j}$ with $\lambda (b + (j -1) e_i , b + je_i)$, where $b$ is as in Theorem~\ref{thm:chain-maps}. 

The aforementioned \cite[Theorem 3.10]{kps-twisted} also shows that the analogue of $\triangledown^{1}_{M}$ for Kumjian, Pask, and Sims was induced by viewing each edge as a $1$-tuple in the obvious way; and indeed, if $\lambda$ is a 1-cube, then $\triangledown_1(\lambda, \mu) = (\lambda, \mu)$.

In order to describe the analogue of $\boxempty^{2}_{M}$ from Theorem 3.16 of \cite{kps-twisted}, we  first recall some   notation used in \cite{kps-twisted}.  Namely, for $\lambda \in \Lambda$, we write 
\[ \lambda = \overline \lambda_1 \cdots\overline{ \lambda}_k\]
where $d(\lambda_i) = d_i(\lambda) e_i$.  With this notation, in \cite{kps-twisted}, Kumjian, Pask, and Sims associated to a composable 2-tuple $(\lambda, \mu)$ the collection of 2-cubes we need to ``flip'' in order to convert $(\overline{\lambda \mu})_1 (\overline{\lambda \mu})_2 \cdots (\overline{\lambda \mu})_{k}$ to  $\overline{\lambda}_1 \overline{\lambda}_{2} \cdots \overline{\lambda}_k \overline{\mu}_1 \cdots \overline{\mu}_k$. The first such 2-cube is given by $\lambda\mu(d_1(\lambda)e_1 , (d_1(\lambda) + d_1(\mu) )e_1+ d_2(\lambda)e_2)$. Given a cubical $2$-cocycle $f \in  Z^2_{\operatorname{cub}} ( \Lambda, M)$, they defined a categorical 2-cocycle $c_f$ such that $c_f(\lambda)$ is given by the sum of the values of $f$ on the aforementioned 2-cubes. 

On the other hand, $\boxempty_2(\lambda, \mu, \nu)$ is the sum of the 2-cubes in $\widetilde{\Lambda}(s(\nu))$ associated to $d_i(\lambda)e_i \times d_{j}(\mu)e_j$ for $j > i$.  These are the 2-cubes we need to ``flip'' in order to convert 
$\left( \underline{\lambda \mu}_k \underline{\lambda \mu}_{k-1} \cdots \underline{\lambda \mu}_1 , \nu \right)$ to 
$\left( \underline{\lambda}_k \underline{\lambda}_{k-1} \cdots \underline{\lambda}_1 \underline{\mu}_k \cdots \underline{\mu_1} , \nu \right)$.  In other words, to pass from Kumjian, Pask, and Sims' procedure to $\boxempty_2$, we need to reverse our choice of ordering on the generators of $\N^k$.

It follows from the proof of \cite[Theorem 4.15]{kps-twisted} that reversing the color order takes a cubical 2-cocycle to its inverse. Thus, using $\boxempty_2$ instead of the procedure from \cite[Theorem 3.16]{kps-twisted} will remove the unfortunate minus sign which appears in the isomorphism  of \cite[Theorem 4.15]{kps-twisted} between cubical and categorical 2-cohomology.

The same Theorem 4.15 establishes that Kumjian, Pask, and Sims used (a version of) the map $\triangledown_2$ to map 2-cubes to composable 2-tuples. More precisely, this follows from the third displayed equation in the proof of \cite[Theorem 4.15]{kps-twisted}.  Here, one sees that if $\lambda = \mu \nu = \nu' \mu'$ where $d(\mu) = d(\mu') = e_i, d(\nu) = d(\nu') = e_j$ and $j > i$, then Kumjian, Pask, and Sims mapped the 2-cube $\lambda \in Q_2(\Lambda)$ to the element   
\[ (\mu, \nu) - (\nu', \mu') = \sum_{\sigma \in \Sigma_2} \operatorname{sgn}(\sigma) \left(\lambda^\sigma_1,  \lambda^\sigma_2\right) \in \Z \Lambda^{*2}.\] 
Consequently,  if we use the   functor $\widetilde \Lambda: \Lambda \to  k \text{\textendash} \mathfrak{graph}$ to translate the Kumjian-Pask-Sims map into a map $\Z \widetilde Q_2(\widetilde{\Lambda}) \to \mathcal P_2(\Lambda)$, we obtain precisely the map $\triangledown_2$.

In terms of cocycles, then, a categorical 2-cocycle $f \in Z^2_{} ( \Lambda, M)$ induces a cubical 2-cocycle $\lambda \mapsto f(\mu, \nu) - f(\nu', \mu')$. This is an analogue of the formula producing an alternating bicharacter from a 2-cocycle on the group $\Z^k$.

\bibliographystyle{amsplain}
\bibliography{eagbib}
\end{document}